\documentclass[11pt,reqno]{amsart}
\usepackage{amsmath,amscd,amssymb,amsfonts,amsthm,courier,relsize,bm,tikz}
\usetikzlibrary{patterns}
\usepackage{hyperref,enumitem,mathrsfs,slashed,mathtools,graphicx}

\usepackage{xcolor}  
\hypersetup{
    colorlinks,
    linkcolor={red!50!black},
    citecolor={blue!70!black},
    urlcolor={blue!80!black}
}

\usepackage[cmtip,matrix,arrow]{xy}

\newtheorem{theorem}{Theorem}[section]
\newtheorem{corollary}[theorem]{Corollary}
\newtheorem{proposition}[theorem]{Proposition}

\theoremstyle{definition}    
\newtheorem{definition}[theorem]{Definition}
\theoremstyle{remark}

\newtheorem{remark}[theorem]{Remark}

\newtheorem{example}[theorem]{Example}

\newcommand{\pair}[2]{\langle #1, #2 \rangle}
\newcommand{\ignore}[1]{}

\newcommand{\ol}[1]{\overline{#1}}
\newcommand{\ul}[1]{\underline{#1}}
\newcommand{\ti}[1]{\widetilde{#1}}
\newcommand{\wh}[1]{\widehat{#1}}
\newcommand{\st}[1]{\mathsf{#1}}
\newcommand{\scr}[1]{\mathscr{#1}}
\newcommand{\sm}[1]{\mathsmaller{#1}}
\newcommand{\tn}[1]{\textnormal{#1}}
\renewcommand{\i}{{\mathrm{i}}}

\def\Ad{\ensuremath{\textnormal{Ad}}}

\def\g{\ensuremath{\mathfrak{g}}}

\def\t{\ensuremath{\mathfrak{t}}}

\def\n{\ensuremath{\mathfrak{n}}}

\def\a{\ensuremath{\mathfrak{a}}}

\def\hvee{\ensuremath{\textnormal{h}^\vee}}

\def\R{\ensuremath{\mathcal{R}}}
\def\C{\ensuremath{\mathcal{C}}}

\def\F{\ensuremath{\mathcal{F}}}
\def\E{\ensuremath{\mathcal{E}}}

\def\A{\ensuremath{\mathcal{A}}}

\def\Y{\ensuremath{\mathcal{Y}}}
\def\U{\ensuremath{\mathcal{U}}}
\def\V{\ensuremath{\mathcal{V}}}

\def\N{\ensuremath{\mathcal{N}}}
\def\H{\ensuremath{\mathcal{H}}}
\def\K{\ensuremath{\mathcal{K}}}

\def\M{\ensuremath{\mathcal{M}}}
\def\S{\ensuremath{\mathcal{S}}}
\def\Q{\ensuremath{\mathcal{Q}}}

\def\c{\ensuremath{\mathsf{c}}}
\def\hc{\ensuremath{\widehat{\mathsf{c}}}}

\def\bC{\ensuremath{\mathbb{C}}}
\def\bR{\ensuremath{\mathbb{R}}}
\def\bZ{\ensuremath{\mathbb{Z}}}

\def\bB{\ensuremath{\mathbb{B}}}
\def\bK{\ensuremath{\mathbb{K}}}

\def\End{\ensuremath{\textnormal{End}}}
\def\Aut{\ensuremath{\textnormal{Aut}}}

\def\Hom{\ensuremath{\textnormal{Hom}}}

\def\pr{\ensuremath{\textnormal{pr}}}

\def\ran{\ensuremath{\textnormal{ran}}}
\def\.{\ensuremath{\cdot}}
\def\ker{\ensuremath{\textnormal{ker}}}

\def\comp{\ensuremath{\textnormal{comp}}}
\def\id{\ensuremath{\textnormal{id}}}

\def\Tr{\ensuremath{\textnormal{Tr}}}

\def\Cl{\ensuremath{\textnormal{Cl}}}
\def\Cliff{\ensuremath{\textnormal{Cliff}}}
\def\index{\ensuremath{\textnormal{index}}}

\def\anti{\ensuremath{\textnormal{anti}}}

\def\dim{\ensuremath{\textnormal{dim}}}

\def\aff{\ensuremath{\textnormal{aff}}}
\def\KK{\ensuremath{\textnormal{KK}}}

\def\pt{\ensuremath{\textnormal{pt}}}
\def\op{\ensuremath{\textnormal{op}}}

\def\K{\ensuremath{\textnormal{K}}}

\def\triv{\ensuremath{\textnormal{triv}}}

\def\rot{\ensuremath{\tn{rot}}}
\def\Irr{\ensuremath{\tn{Irr}}}

\def\bas{\ensuremath{\tn{bas}}}

\def\hGamma{\ensuremath{\widehat{\Gamma}}}
\def\hG{\ensuremath{\widehat{G}}}
\def\hg{\ensuremath{\widehat{g}}}

\def\hgamma{\ensuremath{\widehat{\gamma}}}

\def\hPi{\ensuremath{\widehat{\Pi}}}

\def\heta{\ensuremath{\widehat{\eta}}}

\begin{document}
\sloppy
\title{Geometric K-homology and the Freed-Hopkins-Teleman Theorem}
\author{Yiannis Loizides}

\begin{abstract}
We construct a map at the level of cycles from the equivariant twisted K-homology of a compact, connected, simply connected Lie group $G$ to the Verlinde ring, which is inverse to the Freed-Hopkins-Teleman isomorphism.  As an application, we prove that two of the proposed definitions of the quantization of a Hamiltonian loop group space---one via twisted K-homology of $G$ and the other via index theory on non-compact manifolds---agree with each other.
\end{abstract}
\maketitle

\section{Introduction}
A remarkable theorem due to Freed, Hopkins and Teleman \cite{FHTI,FHTII,FHTIII} relates the representation theory of the loop group $LG$ of a compact Lie group $G$ to the equivariant twisted K-theory of $G$.  In the special case of a connected, simply connected and simple Lie group, the theorem states that there is an isomorphism of rings $R_k(G)\simeq \K^G_0(G,\A^{(k+\hvee)})$.  Here $R_k(G)$ is the Verlinde ring of level $k>0$ positive energy representations of the basic central extension $LG^{\tn{bas}}$ of the loop group, while $\K^G_0(G,\A^{(k+\hvee)})$ is the equivariant K-homology of $G$ with twisting (Dixmier-Douady class) $k+\hvee \in \bZ \simeq H^3_G(G,\bZ)$.  The shift $\hvee$ is a Lie theoretic constant associated to $G$ called the dual Coxeter number.  Freed-Hopkins-Teleman work with twisted K-theory, which is related by a Poincare duality isomorphism \cite{TuPoincare}.

In the proof Freed, Hopkins and Teleman construct a map, at the level of cycles 
\[ \left(\begin{array}{c} \text{level $k$ positive energy}\\ \text{representations of } LG^{\tn{bas}} \end{array}\right) \quad \dashrightarrow \quad \left(\begin{array}{c} \text{$(k+\hvee)$-twisted}\\ \text{K-theory of } G \end{array}\right).  \]
The construction involves an interesting family of algebraic Dirac operators parametrized by the space of connections on a $G$-bundle over $S^1$.  Computing the equivariant twisted K-theory of $G$ using techniques from algebraic topology, they are able to show that their map is an isomorphism.

It is less clear how to construct a map in the opposite direction (from twisted K-homology to $R_k(G)$) \emph{at the level of cycles}.  One goal of this article is to describe such a map for \emph{analytic cycles} or Fredholm modules, which are the cycles for the analytic description of twisted K-homology (cf. \cite{AtiyahKHomology, HigsonRoe, KasparovNovikov}).  A special class of analytic cycles are those which come from Baum-Douglas-type \emph{geometric cycles} (cf. \cite{BaumDouglas, BaumCareyWang}), and we also study the specialization of our map to such cycles and obtain a correspondingly more explicit description.
\ignore{\begin{theorem}
\label{Thm1}
Let $G$ be a compact, connected, simply connected, simple Lie group.  Let $\A$ be a twisting with Dixmier-Douady class $\tn{DD}(\A)=k+\hvee \in \bZ \simeq H^3_G(G,\bZ)$, $k>0$.  The map
\[ \scr{I} \colon \K_0^G(G,\A) \rightarrow R_k(G)\]
defined in detail in Section \ref{sec:DefI}, is inverse to the Freed-Hopkins-Teleman isomorphism.
\end{theorem}}

We should remark at the outset that we do not directly build a positive energy representation from a cycle, which would be interesting and perhaps preferable.  The output will instead be a formal character.  Let us give an overview of the construction.  The data used to describe a twist in the analytic picture is a $G$-equivariant Dixmier-Douady bundle $\A$ over $G$; this is a locally trivial bundle of elementary $C^\ast$ algebras over $G$, equipped with a $G$-action covering the conjugation action on the base.  Such a bundle has an invariant, the Dixmier-Douady class $\tn{DD}(\A) \in H^3_G(G,\bZ)\simeq \bZ$, and we assume $\tn{DD}(\A)=\ell>0$.  Consider an analytic cycle representing a class $x$ in the twisted K-homology group $\K_0^G(G,\A)$.  Restrict $x$ from $G$ to a tubular neighborhood $U$ of a maximal torus $T$ inside $G$.  Over $U$ we show that there is a Morita equivalent Dixmier-Douady bundle $\A_U$ which has an especially simple structure: its algebra of continuous sections can be presented as a twisted crossed product algebra $\Pi \ltimes_\tau C_0(\U)$, where $\tau$ is the twist, $\Pi$ is the integer lattice, and $\U$ is a $\Pi$-covering space of $U$.  Applying tools from KK-theory (a Green-Julg-type isomorphism and the analytic assembly map), we obtain an element in the K-theory of the group $C^\ast$ algebra for $T\ltimes \Pi^\tau$.  There is a map from the latter K-group into the space of formal characters $R^{-\infty}(T)$ for $T$.  The image of $\K_0^G(G,\A)$ under this composition is $R^{-\infty}(T)^{W_\aff-\anti,\,\ell}$, the subspace of formal characters that are alternating under the action of the affine Weyl group at level $\ell$.  For $\ell>\hvee$, the space of such characters is canonically isomorphic to $R_k(G)$, $k=\ell-\hvee$, via the Weyl-Kac character formula.

Our construction provides an elaboration of a remark made by Freed, Hopkins and Teleman in \cite[Remark 3.5]{FHTIII}.  They comment that there ought to be an inverse map from the twisted K-theory of $T$ to a suitable `representation ring' for $T\ltimes \Pi^\tau$ perhaps defined using $C^\ast$ algebras, involving an analogue of `integration over $\t$'.  Our `integration over $\t$' map is the analytic assembly map.

We study the specialization of our map to `D-cycles' in the sense of Baum-Carey-Wang \cite{BaumCareyWang}, which are an analogue of Baum-Douglas cycles in geometric K-homology \cite{BaumDouglas}.  A D-cycle for $\K_0(X,\A)$ is a 4-tuple $(M,E,\Phi,\S)$ consisting of a compact Riemannian manifold $M$, a Hermitian vector bundle $E$ over $M$, a continuous map $\Phi \colon M \rightarrow X$, and a Morita morphism $\S \colon \Cliff(TM) \dashrightarrow \Phi^\ast \A$.  If $\A=\ul{\bC}$ is trivial, $\S$ is equivalent to a spin-c structure on $M$, and we recover an ordinary Baum-Douglas cycle.

In the case of a D-cycle $(M,E,\Phi,\S)$ representing a class $x \in \K^G_0(G,\A)$, we prove that its image under our map is the $T$-equivariant $L^2$-index of a first-order elliptic operator on a $\Pi$-covering space of $\Phi^{-1}(U)$, where $U \supset T$ is a tubular neighborhood of the chosen maximal torus $T$ in $G$.  Let us give a summary of the proof.  Using $(M,E,\Phi,\S)$ we construct an analytic cycle $(H,\rho,\st{D})$ representing $x$: the Hilbert space $H$ is the space of $L^2$ sections of a smooth Hilbert bundle over $M$ and $\st{D}$ is a Dirac operator acting on smooth sections.  Because the bundle has infinite dimensional fibres, $\st{D}$ is not necessarily Fredholm, but the action of the $C^\ast$ algebra $C(\A)$ (continuous sections of $\A$) along the fibres provides the needed analytic control to make this a cycle.  After passing to the Morita equivalent Dixmier-Douady bundle over $U \supset T$, the fibres are replaced with copies of $L^2(\Pi)$ (tensored with a finite dimensional bundle); using a correspondence that is well-known (for example in the context of Atiyah's $L^2$-index theorem), the operator $\st{D}$ then has an alternate interpretation as a Dirac-type operator on a $\Pi$-covering space $\Y$ of $\Phi^{-1}(U)$.  Applying the analytic assembly map gives the $T$-equivariant $L^2$-index of this operator.

The initial motivation for this work was to understand the relationship between two approaches---one via $D$-cycles for twisted K-homology of $G$ \cite{MeinrenkenKHomology} and the other via index theory on non-compact manifolds \cite{LMSspinor,LSQuantLG}---to defining a representation-theoretic `quantization' of a Hamiltonian $LG$-space.  The construction of a suitable quantization has interesting applications, for example to the Verlinde formula for moduli spaces of flat connections on Riemann surfaces, cf. \cite{LecturesGroupValued} for an overview.  A corollary of our results is that the two approaches agree with each other.  Indeed for $x$ represented by a D-cycle, the first-order elliptic operator on $\Y$ mentioned above coincides with the operator studied in \cite{LSQuantLG}.  Our construction thus connects the index of this operator with the image of $x$ in $R_k(G)$ under the Freed-Hopkins-Teleman isomorphism.  
\ignore{
Let $(\M,\omega_{\M},\Phi_{\M})$ be a proper Hamiltonian $LG$-space, that is, a Banach manifold $\M$ equipped with a smooth $LG$-action, a symplectic form, and a proper moment map $\Phi_{\M} \colon \M \rightarrow L\g^\ast$ (see Section \ref{sec:HamLGSpace}).  The subgroup $\Omega G \subset LG$ of based loops acts freely and properly on $L\g^\ast$ and $\M$, and the smooth quotient $M=\M/\Omega G \rightarrow L\g^\ast/\Omega G=G$ is an example of a \emph{quasi-Hamiltonian $G$-space} \cite{AlekseevMalkinMeinrenken}.  It was shown in \cite{DDDFunctor} (\cite{LMSspinor} for an alternate proof) that quasi-Hamiltonian $G$-spaces naturally give rise to D-cycles for the twisted K-homology of $G$.  Based on this, Meinrenken in \cite{MeinrenkenKHomology} defined the quantization of a (pre-quantized) quasi-Hamiltonian $G$-space in terms of a push-forward map to twisted K-homology of $G$.  By the Freed-Hopkins-Teleman theorem, the result can be identified with an element of the Verlinde ring, hence seems a reasonable candidate for the quantization of the $LG$-space $\M$.

In \cite{LMSspinor,LSQuantLG} we constructed a finite dimensional spin-c submanifold $\Y \subset \M$ and then studied an index pairing with the Dirac operator, with the result being a formal character of $T$ alternating under the action of the affine Weyl group.  This suggests that alternatively one might define the quantization of $\M$ as the element of the Verlinde ring whose Weyl-Kac numerator equals this formal character.  
}

Throughout the paper we have restricted ourselves to the special case that $G$ is connected, simply connected and simple, but the methods likely generalize.  We will fairly easily be able to check that the map $\scr{I}\colon \K^G_0(G,\A) \rightarrow R^{-\infty}(T)^{W_\aff-\anti,\ell}$ that we construct is surjective.  With some additional effort and a little topology, together with a known (and relatively easy) case of the Baum-Connes conjecture, we could use the construction described here to prove a weak form of the Freed-Hopkins-Teleman theorem (that $\scr{I}$ is also injective modulo torsion at primes dividing the order of the Weyl group).  We hope to return to these questions in the future.

There is an overlap of some of our methods with interesting work of Doman Takata on Hamiltonian $LT$-spaces \cite{Takata1,Takata2}.  In particular, Takata also studies an assembly map into the K-theory of a twisted group $C^\ast$ algebra of $T \times \Pi$.  Takata has built infinite dimensional analogues of several well-known objects from index theory/non-commutative geometry in the setting of Hamiltonian $LT$-spaces.  It would be interesting to explore further connections with his work.

Sections 2 and 3 briefly introduce twisted K-homology, loop groups, and the Freed-Hopkins-Teleman theorem.  Section 4 contains some results on twisted convolution algebras and generalized fixed-point algebras.  In Section 5 we prove some basic facts about the $C^\ast$ algebra of the semi-direct product $T\ltimes \Pi^\tau$ that plays a key role.  In Section 6 we construct the map, denoted $\scr{I}$, from $\K_0^G(G,\A)$ to $R^{-\infty}(T)^{W_\aff-\anti, \, \ell}$, and prove that it is inverse to the Freed-Hopkins-Teleman map.  Section 7 studies the specialization of $\scr{I}$ to geometric cycles (D-cycles in the sense of Baum-Carey-Wang), and briefly describes the application to Hamiltonian loop group spaces.  For the reader's convenience, we have included an appendix with proofs of a couple of standard (but not so easy to find) results used in Section 7.

\vspace{0.3cm}

\noindent \textbf{Acknowledgements.}
I especially thank Eckhard Meinrenken and Yanli Song for interesting discussions about quantization of Hamiltonian $LG$-spaces over the past couple of years.  The work described in this paper is motivated by our joint work on spinor modules for Hamiltonian $LG$-spaces \cite{LMSspinor,LSQuantLG}.  I also thank Nigel Higson and Shintaro Nishikawa for helpful suggestions and for answering several questions about KK-theory. 

\vspace{0.3cm}

\noindent \textbf{Notation.}
The $C^\ast$ algebras of bounded (resp. compact) operators on a Hilbert space $H$ will be denoted $\bB(H)$ (resp. $\bK(H)$).

If $(V,g)$ is a finite dimensional real Euclidean vector space, $\Cliff(V)$ denotes the complex Clifford algebra of $V$, the $\bZ_2$-graded complex algebra generated in degree 1 by the elements $v \in V$ subject to the relation $v^2=\|v\|^2$.  For $V$ a real Euclidean vector bundle over $M$, $\Cliff(V)$ denotes the bundle of algebras with fibres $\Cliff(V)_m=\Cliff(V_m)$.  On a Riemannian manifold $M$, we write $\Cl(M)$ for the algebra of continuous sections of $\Cliff(TM)$ vanishing at infinity.

If $K$ is a compact Lie group, $\Irr(K)$ denotes the set of isomorphism classes of irreducible representations of $K$, and $R(K)$ is the representation ring.  The formal completion $R^{-\infty}(K)=\bZ^{\Irr(K)}$ consists of formal infinite linear combinations of irreducibles $\pi \in \Irr(K)$ with coefficients in $\bZ$.  When discussing a $U(1)$ central extension $\Gamma^\tau$ of a group $\Gamma$, we use the notation $\hgamma$ to denote some lift to $\Gamma^\tau$ of an element $\gamma \in \Gamma$.

Throughout $G$ denotes a compact, connected, simply connected, simple Lie group with Lie algebra $\g$.  Let $T \subset G$ be a maximal torus with Lie algebra $\t$.  We fix a positive Weyl chamber $\t_+$, and let $\R_+$ (resp. $\R_-$) denote the positive (resp. negative) roots.  The half sum of the positive roots is denoted $\rho$, and $\hvee$ is the dual Coxeter number of $G$.  Since $G$ is simply connected, the integer lattice $\Pi=\ker(\exp \colon \t \rightarrow T)$ coincides with the coroot lattice.  The dual $\Pi^\ast=\Hom(\Pi,\bZ)$ is the (real) weight lattice.  There is a unique invariant inner product $B$ on $\g$, the \emph{basic} inner product, with the property that the squared length of the short co-roots is $2$.  We often use the basic inner product to identify $\g \simeq \g^\ast$, and we sometimes write $B^\flat$, $B^\sharp$ for the musical isomorphisms when we want to emphasize this.  The basic inner product has the property that $B(\Pi,\Pi)\subset \bZ$, and thus $B^\flat(\Pi)\subset \Pi^\ast$.

\section{Twisted K-homology}
Here we give a brief introduction to the analytic description of twisted K-homology.  Our discussion is similar to \cite{MeinrenkenKHomology, MeinrenkenConjugacyClasses} where one can find further details.  For further background on analytic K-homology and KK-theory, see for example \cite{HigsonRoe, HigsonPrimer, KasparovNovikov}.  We also recall the definition of `D-cycles' due to Baum, Carey and Wang \cite{BaumCareyWang}, which are a version of Baum-Douglas-type geometric cycles \cite{BaumDouglas} for twisted K-homology.  

Let $X$ be a locally compact space.  A \emph{Dixmier-Douady bundle} over $X$ is a locally trivial bundle of $C^\ast$ algebras $\A \rightarrow X$, with typical fibre isomorphic to the compact operators $\bK(H)$ for a (separable) Hilbert space $H$, and structure group the projective unitary group $PU(H)$ with the strong operator topology.  Restricting to a sufficiently small open $U \subset X$, $\A|_U$ is isomorphic to $\bK(\H)$ for a bundle of Hilbert spaces $\H \rightarrow U$, but this need not be true globally.  The notation $\A^{\op}$ denotes the Dixmier-Douady bundle obtained by taking the opposite algebra structure on the fibres.  The tensor product $\A_0 \otimes \A_1$ of Dixmier-Douady bundles is again a Dixmier-Douady bundle.

A \emph{Morita morphism} $\S \colon \A_0 \dashrightarrow \A_1$ between Dixmier-Douady bundles over $X$ is a bundle of $\A_1 \otimes \A_0^{\op}$ modules $\S \rightarrow X$, locally modelled on the $\bK(H_1)-\bK(H_0)$ bimodule $\bK(H_0,H_1)$.  In the special case $\A_1=\underline{\bC}$, $\S$ is called a \emph{Morita trivialization} of $\A_0$.  Any two Morita morphisms $\A_0 \dashrightarrow \A_1$ are related by tensoring with a line bundle; if the line bundle is trivial, one says the Morita morphisms are \emph{2-isomorphic}.

By a theorem of Dixmier and Douady \cite{DixmierDouady}, Morita isomorphism classes of Dixmier-Douady bundles are classified by a degree 3 integral cohomology class $\tn{DD}(\A) \in H^3(X,\bZ)$ known as the \emph{Dixmier-Douady class}.  The Dixmier-Douady class satisfies
\[ \tn{DD}(\A^{\op})=-\tn{DD}(\A), \qquad \tn{DD}(\A_0 \otimes \A_1)=\tn{DD}(\A_0)+\tn{DD}(\A_1).\]
There are modest generalizations to the case where the fibres of $\A$ (resp. $\S$) carry $\bZ_2$-gradings; in this case $\A$ (resp. $\S$) is locally modelled on $\bK(H)$ for a $\bZ_2$-graded Hilbert space $H$ (resp. $\bK(H_0,H_1)$ with $H_0$, $H_1$ being $\bZ_2$-graded Hilbert spaces), and the Dixmier-Douady class $\tn{DD}(\A) \in H^3(X,\bZ)\oplus H^1(X,\bZ_2)$.  If $X$ carries an action of a compact group $G$, one can define $G$-equivariant Dixmier-Douady bundles, which are classified up to $G$-equivariant Morita morphisms by classes in the analogous equivariant cohomology groups.

The $C^\ast$ algebraic definition of twisted K-theory goes back to Donovan-Karoubi \cite{DonovanKaroubi} (in the case of a torsion Dixmier-Douady class) and Rosenberg \cite{RosenbergContinuousTrace} (the general case); see also \cite{AtiyahSegalTwistedK,KaroubiOldNew}.  Let $\A$ be a $G$-equivariant $\bZ_2$-graded Dixmier-Douady bundle and $C_0(\A)$ the $\bZ_2$-graded $G$-$C^\ast$ algebra of continuous sections of $\A$ vanishing at infinity.  One defines the $G$-equivariant $\A$-\emph{twisted K-homology} of $X$ to be the $G$-equivariant analytic K-homology of this $C^\ast$ algebra:
\[ \K_i^G(X,\A)=\KK^i_G(C_0(\A),\bC), \qquad i=0,1\]
where $\KK_G^i(A,B)$ is Kasparov's KK-theory (cf. \cite{KasparovNovikov}).  This definition is well known to be equivalent to Atiyah-Segal's \cite{AtiyahSegalTwistedK} description in terms of homotopy classes of continuous sections of bundles with typical fibre the Fredholm operators on a Hilbert space.

\begin{remark}
\label{rem:CanonicalIso}
A Morita morphism $\A_0 \dashrightarrow \A_1$ defines an invertible element in the group $\KK^0_G(C_0(\A_1),C_0(\A_0))$, and hence an isomorphism between the corresponding twisted K-homology groups.  Thus, the resulting groups depend only on the Dixmier-Douady class of $\A$.  Note however that there may be no \emph{canonical} isomorphism; different Morita morphisms can lead to genuinely different maps.  Any two Morita morphisms are related by tensoring with a $\bZ_2$-graded line bundle, hence the set of Morita morphisms is a torsor for $H^2_G(X,\bZ)\times H^0_G(X,\bZ_2)$.  
\end{remark}

\begin{example}
\label{ex:DeRhamDirac}
An important example of a $\bZ_2$-graded Dixmier-Douady bundle is the Clifford algebra bundle $\Cliff(TM)$ of a Riemannian manifold $M$.  Kasparov's fundamental class $[\scr{D}]$ is the class in the twisted K-homology group $\K_0(M,\Cliff(TM))=\KK(\Cl(M),\bC)$ represented by the de Rham-Dirac operator $\scr{D}=d+d^\ast$ acting on smooth differential forms over $M$ (cf. \cite[Definition 4.2]{KasparovNovikov}).  A Morita trivialization $\S \colon \Cliff(TM)\dashrightarrow \underline{\bC}$ is the same thing as a spinor module for $\Cliff(TM)$.  $\S$ defines an invertible element $[\S] \in \KK(C_0(M),\Cl(M))$, and the KK product $[\S]\otimes_{\Cl(M)} [\scr{D}] \in \KK(C_0(M),\bC)$ is the class represented by a spin-c Dirac operator for $\S$.  More generally, twisting $\scr{D}$ by a complex vector bundle $E$, one obtains a class $[\scr{D}^E] \in \KK(\Cl(M),\bC)$, and the KK product $[\S]\otimes_{\Cl(M)}[\scr{D}^E]$ is the class represented by the Dirac operator coupled to $E$.
\end{example}

\subsection{Geometric twisted K-homology}
Baum, Carey and Wang \cite{BaumCareyWang} describe a `geometric' approach to twisted K-homology, in the spirit of Baum-Douglas \emph{geometric K-homology} \cite{BaumDouglas} (see also \cite{BaumHigsonSchick}).  Actually in \cite{BaumCareyWang} two types of cycles for twisted geometric K-homology are discussed: `K-cycles' versus `D-cycles'.  The geometric K-homology groups defined by both types of cycles admit natural maps to the analytic K-homology group described above.  In this paper we will only discuss D-cycles, and only use the even case.
\begin{definition}\cite{BaumCareyWang}
Let $\A$ be a $G$-equivariant $\bZ_2$-graded Dixmier-Douady bundle over a locally finite $G$-CW complex $X$.  An (even) \emph{D-cycle} for $(X,\A)$ is a 4-tuple $(M,E,\Phi,\S)$ where
\begin{itemize}
\item $M$ is an even-dimensional smooth closed $G$-manifold, with a $G$-invariant Riemannian metric
\item $\Phi \colon M \rightarrow X$ is a $G$-equivariant continuous map
\item $E$ is a $G$-equivariant Hermitian vector bundle over $M$
\item $\S \colon \Cliff(TM) \dashrightarrow \Phi^\ast \A$ is a $G$-equivariant Morita morphism.
\end{itemize}
\end{definition}
\begin{remark}
The terminology `D-cycle' comes from string theory.  If $M$ is orientable, the Dixmier-Douady class of $\Cliff(TM)$ is the third integral Stieffel-Whitney class $W_3(M)$ (the obstruction to the existence of a spin-c structure on $M$).  The existence of $\S$ implies
\[ \Phi^\ast \tn{DD}(\A)=W_3(M),\]
which is called the `Freed-Witten anomaly cancellation condition' in string theory.
\end{remark}

The \emph{geometric twisted K-homology} $\K^G_{\tn{geo},i}(X,\A)$ of $X$ is the set of D-cycles modulo an equivalence relation analogous to Baum-Douglas geometric K-homology (generated by suitable versions of `disjoint union equals direct sum', `bordism', and `bundle modification'), see \cite{BaumCareyWang}.  There is a functorial map
\begin{equation} 
\label{eqn:GeoAnalyticMap}
\K^G_{\tn{geo},i}(X,\A) \rightarrow \K^G_i(X,\A)
\end{equation}
which is straight-forward to describe at the level of cycles.  We will only use the even case $i=0$ here; the odd case is similar.  Let $[\scr{D}^E] \in \KK_G(\Cl(M),\bC)$ be the class of the de Rham-Dirac operator on $M$, coupled to the vector bundle $E$.  The pair $(\Phi,\S)$ defines a push-forward map 
\[ (\Phi,\S)_\ast \colon \KK_G(\Cl(M),\bC) \rightarrow \KK_G(C_0(\A),\bC) \]
given as the composition of the Morita morphism $\Cliff(TM) \dashrightarrow \Phi^\ast \A$, with the map induced by the $^\ast$-homomorphism
\[ \Phi^\ast \colon C_0(\A) \rightarrow C_0(\Phi^\ast \A).\]
The image of $[(M,E,\Phi,\S)]$ in $\K^G_0(X,\A)$ is the push-forward:
\begin{equation} 
\label{eqn:PushNotation}
(\Phi,\S)_\ast [\scr{D}^E].
\end{equation}
The push-forward can alternately be expressed as a KK product
\begin{equation} 
\label{eqn:ProdNotation}
[\S] \otimes [\scr{D}^E] 
\end{equation}
where $[\S]\in \KK_G(C_0(\A),\Cl(M))$ is the class defined by the triple $(C_0(\S),\Phi^\ast,0)$.

\begin{remark}
A proof that the map \eqref{eqn:GeoAnalyticMap} is an isomorphism has been announced by Baum, Joachim, Khorami and Schick \cite{BJKS}, at least for the non-equivariant case.
\end{remark}

\subsection{Twisted K-homology of $G$.}
Let $G$ be a compact, connected, simply connected, simple Lie group acting on itself by conjugation.  Then $H^3_G(G,\bZ)\simeq \bZ$, while $H^2_G(G,\bZ)=H^1_G(G,\bZ_2)=H^0_G(G,\bZ_2)=0$.  There is a canonical generator of $H^3_G(G,\bZ)$; in de Rham cohomology, it is represented by the equivariant Cartan 3-form
\[ \eta_G(\xi)=\eta-\tfrac{1}{2}B(\theta^L+\theta^R,\xi), \qquad \eta=\tfrac{1}{12}B(\theta^L, [\theta^L,\theta^L]),\]
where $\xi \in \g$ and $\theta^L$ (resp. $\theta^R$) denotes the left (resp. right) invariant Maurer-Cartan form.  Thus $G$-equivariant Dixmier-Douady bundles $\A$ over $G$ are classified up to Morita equivalence by an integer $\ell \in \bZ$, and moreover any two Morita morphisms are 2-isomorphic, see Remark \ref{rem:CanonicalIso}.  Although we will not use it, it is known that the twisted K-homology group $\K^G_0(G,\A)$ carries a ring structure; in this picture the ring structure originates from a canonical Morita morphism $\A \boxtimes \A \dashrightarrow \tn{Mult}^\ast \A$, where $\tn{Mult}\colon G \times G \rightarrow G$ is the group multiplication, cf. \cite{FHTI,MeinrenkenConjugacyClasses}.  

\section{Loop groups and the Freed-Hopkins-Teleman Theorem}\label{sec:loopgroup}
In this section we briefly introduce the loop group $LG$ and its important class of projective positive energy representations, cf. \cite{PressleySegal}.  We then recall the Freed-Hopkins-Teleman theorem, which relates loop groups to twisted K-homology.

To obtain a Banach-Lie group, we take $LG$ to consist of maps $S^1=\bR/\bZ \rightarrow G$ of some fixed Sobolev level $s>\tfrac{1}{2}$.  The basic inner product defines a central extension of the Lie algebra $L\g$ by $\bR$, with cocycle
\begin{equation} 
\label{eqn:cocycleLG}
c(\xi_1,\xi_2)=\int_0^1 B(\xi_1(t), \xi_2^\prime(t))\,\,dt.
\end{equation}
This extension integrates to a $U(1)$ central extension $LG^\bas$ of $LG$, that we will call the \emph{basic central extension}.  For $G$ connected, simple, and simply connected, $U(1)$ central extensions of $LG$ are uniquely determined by their Lie algebra cocycle, which must be an integer multiple of the generator \eqref{eqn:cocycleLG}; thus $U(1)$ central extensions are classified by $\bZ$, with $LG^\bas$ corresponding to the generator $1 \in \bZ$.
\ignore{
$U(1)$ central extensions of $LG$ are classified by an integer $k$ called the \emph{level}, and the basic central extension corresponds to the generator $k=1$.}

For later reference, note that the loop group can be written as a semi-direct product $LG=G\ltimes \Omega G$, where $\Omega G=\{\gamma \in LG|\gamma(0)=\gamma(1)\}$ is the based loop group, and $G \hookrightarrow LG$ identifies $G$ with the constant loops in $LG$.  Our assumptions on $G$ imply that any $U(1)$ central extension of $G$ is trivial, hence in particular the restriction of $LG^\bas$ to the constant loops is trivial.

Let $T \subset G$ be a fixed maximal torus and let $\Pi=\ker(\exp \colon \t \rightarrow T)$ be the integral lattice.  The product group $T \times \Pi$ may be viewed as a subgroup of $LG$, where $T$ is embedded as constant loops and $\Pi$ as exponential loops: $\eta \in \Pi$ corresponds to the loop $s \in \bR/\bZ \mapsto \exp(s \eta) \in T$.  The restriction of the central extension $LG^\bas$ to $T \times \Pi$ is a central extension
\[ 1 \rightarrow U(1) \rightarrow T\ltimes \Pi^\bas \rightarrow T \times \Pi \rightarrow 1.\]
We discuss the subgroup $T\ltimes \Pi^\bas \subset LG^\bas$ in detail in Section \ref{sec:SemiDirect}.

\subsection{Positive energy representations.}
The loop group has a much-studied class of projective representations known as \emph{positive energy} representations, which have a detailed theory parallel to the theory for compact groups cf. \cite{KacBook,PressleySegal}.  Let $S^1_{\rot} \ltimes LG$ denote the semi-direct product constructed from the action of $S^1$ on $LG$ by rigid rotations.  This action lifts to an action on the basic central extension.  A positive energy representation is a representation of $LG^{\bas}$ which extends to a representation of the semi-direct product $S^1_{\rot}\ltimes LG^{\bas}$ such that the weights of $S^1_{\rot}$ are bounded below.  One can always tensor a positive energy representation by a 1-dimensional representation of $S^1_{\rot}$, hence one often normalizes positive energy representations by requiring that the minimal $S^1_{\rot}$ weight is $0$, and we always assume this.  

For an irreducible positive energy representation, the central circle of $LG^{\bas}$ acts by a fixed weight $k \ge 0$ called the \emph{level}.  There are finitely many irreducible positive energy representations at any fixed level, parametrized by the `level $k$ dominant weights': weights $\lambda \in \Pi^\ast \cap \t^\ast_+$ satisfying $B(\lambda,\theta) \le k$, where $\theta \in \R_+$ is the highest root of $\g$.  Equivalently the level $k$ weights $\Pi^\ast_k=\Pi^\ast \cap k\a$, where $\a \subset \t_+$ is the fundamental alcove, which we identify with a subset of $\t^\ast$ using the basic inner product.

Let $R_k(G)$ denote the free abelian group of rank $\#(\Pi_k^\ast)$ generated by $\bZ$-linear combinations of the level $k$ irreducible positive energy representations.  There is a canonical isomorphism (`holomorphic induction', cf. \cite{FHTI})
\[ R_k(G) \simeq R(G)/I_k(G) \]
where $R(G)$ is the representation ring of $G$ and $I_k(G)$ is the \emph{Verlinde ideal} consisting of characters vanishing on the conjugacy classes of the elements
\[ \exp\Big(\frac{\xi+\rho}{k+\hvee}\Big), \qquad \xi \in \Pi_k^\ast.\]
In particular $R_k(G)$ is a ring, known as the level $k$ \emph{Verlinde ring}.

There is an alternate description of $R_k(G)$ that will be crucial for us later on; this description plays a significant role in the proof of the Freed-Hopkins-Teleman Theorem as well.  An element of $R_k(G)$ is uniquely determined by its \emph{multiplicity function}, a map
\[ m \colon \Pi^\ast_k \rightarrow \bZ.\]
It is known that $\Pi^\ast_k$ is precisely the set of weights contained in the interior of the shifted, scaled alcove $(k+\hvee)\a-\rho$.  The latter is a fundamental domain for the $\rho$-shifted level $(k+\hvee)$ action of the affine Weyl group $W_{\aff}=W\ltimes \Pi$, given by
\begin{equation} 
\label{eqn:ShiftedAction}
w\bullet_{k+\hvee} \xi=(\ol{w},\eta) \bullet_{k+\hvee} \xi=\ol{w}(\xi+\rho)-\rho+(k+\hvee)\eta, \qquad \xi \in \t^\ast, \quad \ol{w} \in W, \eta \in \Pi.
\end{equation}
Thus, $m$ has a \emph{unique} extension to a map
\[ m \colon \Pi^\ast \rightarrow \bZ \]
which is \emph{alternating} under \eqref{eqn:ShiftedAction}, i.e.
\[ m(w\bullet_{k+\hvee} \xi)=(-1)^{l(w)}m(\xi),\]
where $l(w)$ is the length of the affine Weyl group element $w$.  The extension of $m$ vanishes on the boundary of the fundamental domain $(k+\hvee)\a-\rho$.  This defines an isomorphism of abelian groups
\begin{equation} 
\label{eqn:AlternatingFormal}
R_k(G) \xrightarrow{\sim} R^{-\infty}(T)^{W_{\aff}-\anti,\,(k+\hvee)}
\end{equation}
where the right hand side denotes the formal characters of $T$ which are alternating under the action \eqref{eqn:ShiftedAction}.

That \eqref{eqn:AlternatingFormal} is an isomorphism can be deduced more or less immediately from the Weyl-Kac character formula (cf. \cite{KacBook,PressleySegal}).  A positive energy representation has a formal character $\chi \in R^{-\infty}(S^1_{\rot}\times T)$ given by a formula analogous to the Weyl character formula for compact Lie groups, but with the numerator and denominator both being formal infinite expressions.  As in the Weyl character formula, the denominator $\Delta$ is a universal expression (the same for any $\chi \in R_k(G)$); multiplying $\chi$ by $\Delta$ and then restricting to $q=1 \in S^1_{\rot}$, one obtains an element $(\chi \cdot \Delta)|_{q=1} \in R^{-\infty}(T)^{W_{\aff}-\anti,\,(k+\hvee)}$, and this correspondence is one-one.

\subsection{Dixmier-Douady bundles from positive energy representations.}\label{sec:DDoverG}
Loop groups are closely related to twisted K-theory of $G$.  One manifestation of this is that positive energy representations can be used to construct Dixmier-Douady bundles over $G$.

Let $\ell>0$ and let $LG^\tau$ denote the central extension of $LG$ corresponding to $\ell$ times the basic inner product ($\ell=1$ corresponds to $LG^{\bas}$).  Let $V$ be a level $\ell$ positive energy representation, or in other words, a positive energy representation of $LG^\tau$ such that the central circle acts with weight $1$.  The dual space $V^\ast$ carries a negative energy representation such that the central circle acts with weight $-1$.  Let $PG$ denote the space of quasi-periodic paths in $G$ of Sobolev level $s > \tfrac{1}{2}$, that is, $PG$ is the space of paths $\gamma \colon \bR \rightarrow G$ such that $\gamma(t)\gamma(t+1)^{-1}$ is a fixed element of $G$, independent of $t \in \bR$.  The group $LG\times G$ acts on $PG$, with $LG$ acting by right multiplication, and $G$ by left multiplication (cf. \cite{LMSspinor} for further discussion).  The map
\[ q \colon \gamma \in PG \mapsto \gamma(t)\gamma(t+1)^{-1} \in G \]
makes $PG$ into a $G$-equivariant principal $LG$-bundle over $G$.  The adjoint action of $LG^\tau$ on the algebra of compact operators $\bK(V^\ast)$ descends to an action of $LG$, and the associated bundle
\begin{equation} 
\label{eqn:defA}
\A=PG \times_{LG} \bK(V^\ast) 
\end{equation}
is a Dixmier-Douady bundle over $G$ such that $\tn{DD}(\A)=\ell \in \bZ \simeq H^3_G(G,\bZ)$, cf. \cite{MeinrenkenConjugacyClasses}.

\subsection{The Freed-Hopkins-Teleman theorem.}
The following is a special case (for $G$ connected, simply connected, simple) of the Freed-Hopkins-Teleman theorem.
\begin{theorem}[Freed-Hopkins-Teleman \cite{FHTI,FHTII,FHTIII}]
Let $k>0$, and let $\hvee$ be the dual Coxeter number of $G$.  Let $\A$ be a $G$-equivariant Dixmier-Douady bundle over $G$ with $\tn{DD}(\A)=k+\hvee \in \bZ \simeq H^3_G(G,\bZ)$.  The group $\K^G_1(G,\A)=0$, and there is an isomorphism of rings
\[ R_k(G) \simeq \K^G_0(G,\A).\]
\end{theorem}

Let $\iota \colon \{e \} \hookrightarrow G$ be the inclusion of the identity element in $G$.  Consider the model \eqref{eqn:defA} for $\A$.  The Hilbert space $V^\ast$ gives a (canonical) $G$-equivariant Morita trivialization of $\iota^\ast \A$.  Freed-Hopkins-Teleman prove that their isomorphism $R_k(G) \rightarrow \K_0^G(G,\A)$
moreover fits into a commutative diagram
\begin{equation}
\label{diagram:FHT}
\begin{CD}
R(G) @>>> R_k(G)\\
@V \simeq VV      @V \simeq VV\\
\K_0^G(\pt)@>(\iota,V^\ast)>>\K_0^G(G,\A)   
\end{CD}
\end{equation}
where the top horizontal arrow is the quotient map and the bottom horizontal arrow is induced by the evaluation map $\iota^\ast \colon C(\A) \rightarrow \A|_{e}$ composed with the Morita trivialization $V^\ast \colon \A|_{e} \dashrightarrow \bC$. 

\section{Crossed products and twisted K-homology} \label{sec:CrossProdDD}
In this section we describe some general facts involving crossed product algebras, central extensions, and generalized fixed-point algebras.  Throughout this section $\Gamma$, $S$, $N$ are locally compact, second countable topological groups equipped with left Haar measure, and $A$ is a separable $C^\ast$ algebra.

\subsection{Twisted crossed-products.} \label{sec:TwistedCross}
Let $\Gamma$ be a locally compact group with left invariant Haar measure, and let $\Gamma^\tau$ be a $U(1)$-central extension:
\[ 1 \rightarrow U(1) \rightarrow \Gamma^\tau \rightarrow \Gamma \rightarrow 1.\]
Normalize Haar measure on $\Gamma^\tau$ such that the integral of a function over $\Gamma^\tau$ is given by first averaging over $U(1)$ (using normalized Haar measure) followed by integration over $\Gamma$.  A choice of section $\Gamma \rightarrow \Gamma^\tau$ is not needed.  In detail, for $f \in C_c(\Gamma^\tau)$ let
\begin{equation} 
\label{eqn:AvgU1}
\bar{f}(\hgamma)=\int_{U(1)} f(z\hgamma)\, dz.
\end{equation}
Then $\bar{f}$ is a $U(1)$-invariant function on $\hGamma$ so descends to a function on $\Gamma$, and
\begin{equation} 
\label{eqn:IntGamma}
\int_{\Gamma^\tau} f(\hgamma)\, d\hgamma=\int_\Gamma \bar{f}(\gamma)\,d\gamma.
\end{equation}

Let $A$ be a $\Gamma$-$C^\ast$ algebra.  Note that $A$ can be regarded as a $\Gamma^\tau$-$C^\ast$ algebra such that the central circle in $\Gamma^\tau$ acts trivially.  The (maximal) crossed product algebra $\Gamma^\tau \ltimes A=C^\ast(\Gamma^\tau,A)$ (we use both notations interchangeably) decomposes into a direct sum of its homogeneous ideals
\begin{equation} 
\label{eqn:Homogeneous}
\Gamma^\tau \ltimes A=\bigoplus_{n \in \bZ} (\Gamma^\tau \ltimes A)_{(n)} 
\end{equation}
where $(\Gamma^\tau \ltimes A)_{(n)}$ denotes the norm closure (in the maximal crossed product algebra $\Gamma^\tau \ltimes A$) of the set of compactly supported functions $a \in C_c(\Gamma^\tau,A)$ satisfying
\[ a(z^{-1}\wh{\gamma})=z^na(\wh{\gamma}), \qquad z \in U(1), \wh{\gamma} \in \Gamma^\tau.\]
There is a $^\ast$-homomorphism from $C^\ast(U(1))$ into the multiplier algebra $M(\Gamma^\tau \ltimes A)$ (cf. \cite[II.10.3.10-12]{BlackadarCAlgebras}) making $\Gamma^\tau \ltimes A$ into a $C^\ast(U(1))=C_0(\bZ)$-algebra, and the ideals $(\Gamma^\tau \ltimes A)_{(n)}$ are the fibres.  The decomposition \eqref{eqn:Homogeneous} is also not difficult to prove directly.  A short calculation using \eqref{eqn:AvgU1} shows that the $(\Gamma^\tau \ltimes A)_{(n)}$ are 2-sided ideals, and hence one has a $^\ast$-homomorphism from the right hand side of \eqref{eqn:Homogeneous} to $\Gamma^\tau \ltimes A$.  One also has a $^\ast$-homomorphism in the opposite direction, given by `taking Fourier coefficients'.  For further details see for example \cite[Proposition 3.2]{LaurentTuXuDiffStacks} or \cite{Takata1}.

\begin{definition}
\label{def:TwistedCross}
We define the $\tau$-\emph{twisted crossed product algebra} $\Gamma \ltimes_{\tau} A$ to be the ideal
\[ \Gamma \ltimes_\tau A=(\Gamma^\tau \ltimes A)_{(1)}.\]
The special case $A=\bC$ gives the twisted group $C^\ast$ algebra
\[ C^\ast_\tau(\Gamma)=C^\ast(\Gamma^\tau)_{(1)}.\]
\end{definition}
\begin{remark}
\label{rem:ChooseSection}
One often sees the twisted crossed-product algebra defined in terms of a cocycle for the central extension, cf. \cite{MarcolliMathaiTwistedGroup}.  One can translate to this definition by choosing a section $\Gamma \rightarrow \Gamma^\tau$.  One reason we take the approach above is that later on we will consider the action of a second group $S \circlearrowright \Gamma \ltimes_\tau A$, and it seems slightly awkward to describe this in terms of a section $\Gamma \rightarrow \Gamma^\tau$; for example, it is not clear to us that one can find an $S$-invariant section.
\end{remark}

The twisted crossed product algebra $\Gamma \ltimes_\tau A$ has the important universal property that non-degenerate $^\ast$-representations of $\Gamma \ltimes_\tau A$ are in 1-1 correspondence with \emph{covariant pairs} $(\pi_A,\pi_\Gamma^\tau)$, where $\pi_A$ is a $^\ast$-representation of $A$, $\pi_\Gamma^\tau$ is representation of $\Gamma^\tau$ such that the central circle acts with weight $1$ (a $\tau$-projective representation of $\Gamma$), and
\begin{equation} 
\label{eqn:covpair}
\pi_\Gamma^\tau(\wh{\gamma})\pi_A(a)\pi_\Gamma^\tau(\wh{\gamma})^{-1}=\pi_A(\hgamma \cdot a)
\end{equation}
for all $\hgamma \in \Gamma^\tau$, $a \in A$.

The space $L^2(\Gamma^\tau)$ splits into an $\ell^2$-direct sum of its homogeneous subspaces
\begin{equation} 
\label{eqn:L2HomogSub}
L^2(\Gamma^\tau)=\bigoplus_{n \in \bZ} L^2(\Gamma^\tau)_{(n)},
\end{equation}
where $L^2(\Gamma^\tau)_{(n)}$ denotes the subspace of $L^2(\Gamma^\tau)$ consisting of functions $f \in L^2(\Gamma^\tau)$ satisfying
\[ f(z^{-1}\hgamma)=z^nf(\hgamma), \qquad z \in U(1), \hgamma \in \Gamma^\tau.\]
Recall the left and right regular representations of $\Gamma^\tau$ on $L^2(\Gamma^\tau)$ are given, respectively, by
\[ \lambda(\hgamma)f(\hgamma_1)=f(\hgamma^{-1}\hgamma_1), \qquad \rho(\hgamma)f(\hgamma_1)=f(\hgamma_1\hgamma).\]
Both actions preserve the decomposition \eqref{eqn:L2HomogSub}.
\begin{definition}
\label{def:TwistedReg}
The \emph{left} $\tau$-\emph{twisted regular representation} of $\Gamma$ is the restriction of the left regular representation of $\Gamma^\tau$ on $L^2(\Gamma^\tau)$ to the subspace
\[ L^2_{\tau}(\Gamma):=L^2(\Gamma^\tau)_{(1)}.\]
The restriction of the right regular representation to $L^2_{\tau}(\Gamma)$ is the \emph{right} $(-\tau)$-\emph{twisted regular representation} of $\Gamma$.  (Note that under the right regular representation, the central circle of $\Gamma^\tau$ acts on $L^2_{\tau}(\Gamma)$ with weight $-1$.)
\end{definition}

\subsection{Dixmier-Douady bundles from crossed-products.}\label{sec:DDbdleCrossProd}
Let $X$ be a locally compact Hausdorff space with a continuous proper action of a locally compact group $\Gamma$.  The quotient $X/\Gamma$ equipped with the quotient topology is then also a locally compact Hausdorff space.  Let $\Gamma$ act on $L^2(\Gamma)$ by right translation, and on $\bK:=\bK(L^2(\Gamma))$ by the adjoint action.  Define the algebra of sections of a field of $C^\ast$-algebras over $X/\Gamma$, suggestively denoted $C_0(X\times_\Gamma \bK)$, consisting of $\Gamma$-equivariant continuous maps $X \rightarrow \bK$ vanishing at infinity in $X/\Gamma$.  The algebra $C_0(X\times_\Gamma \bK)$ is an example of a \emph{generalized fixed-point algebra}.

The following result is attributed to Rieffel (for example \cite[Proposition 4.3]{RieffelFiniteGroup}, \cite{RieffelProperActions}); see especially \cite[Corollary 2.11]{EchterhoffEmerson} for a statement formulated in the same terms used here.  Another reference is \cite[Proposition 4.3]{LaurentTuXuDiffStacks} where a quite general statement appears for twisted convolution algebras of locally compact proper groupoids.
\begin{proposition}
\label{prop:RieffelFixedPt}
Let $X$ be a locally compact Hausdorff space with a continuous proper action of a locally compact group $\Gamma$, and let $\bK=\bK(L^2(\Gamma))$.  Then
\[ \Gamma \ltimes C_0(X)\simeq C_0(X\times_\Gamma \bK).\]
\end{proposition}
\begin{remark}
\label{rem:RieffelMap}
We mention briefly how a map $\Gamma \ltimes C_0(X) \rightarrow C_0(X\times_\Gamma \bK)$ is constructed.  Using conventions as in \cite[Section 3.7]{KasparovNovikov}, a function $a \in C_c(\Gamma,C_c(X))$ is sent to the $\Gamma$-equivariant family $K_a(x) \in \bK$, $x \in X$ of compact operators defined by the family of integral kernels
\begin{equation} 
\label{eqn:RieffelIso}
k_a(\gamma_1,\gamma_2;x)=\mu(\gamma_2)^{-1}a(\gamma_1\gamma_2^{-1};\gamma_1x)
\end{equation}
where $\mu \colon \Gamma \rightarrow \bR_{>0}$ is the modular homomorphism of $\Gamma$.
\end{remark}
\begin{remark}
\label{rem:LeftReg}
Proposition \ref{prop:RieffelFixedPt} can be viewed as a generalization of the Stone-von Neumann theorem (obtained from the special case $\Gamma=\bR$ acting on $X=\bR$ by translations).  More generally for $X=\Gamma$, Proposition \ref{prop:RieffelFixedPt} specializes to a well-known isomorphism
\begin{equation} 
\label{eqn:SpecialCaseRieffel}
\Gamma \ltimes C_0(\Gamma) \xrightarrow{\sim} \bK(L^2(\Gamma)).
\end{equation}
Using equation \eqref{eqn:RieffelIso} one verifies that the induced map on multiplier algebras sends $C_0(\Gamma)$ to multiplication operators and $\Gamma$ to the \emph{left} regular representation.
\end{remark}

If the action of $\Gamma$ on $X$ is free, then $X \rightarrow X/\Gamma$ is a principal $\Gamma$-bundle, and the generalized fixed-point algebra is the algebra of continuous sections vanishing at infinity of the associated bundle
\begin{equation}
\label{eqn:BoringDD} 
\A=X \times_\Gamma \bK.
\end{equation}
This is a Dixmier-Douady bundle, with typical fibre $\bK(L^2(\Gamma))$.  In fact $\A$ is Morita trivial with Morita trivialization $X \times_\Gamma L^2(\Gamma)$.  

To obtain something more interesting from the construction \eqref{eqn:BoringDD}, we adjust it slightly in two ways.  First we consider the equivariant situation, where a second group $S$ acts on $X$ and $L^2(\Gamma)$.  It may happen that the Morita trivialization $X \times_\Gamma L^2(\Gamma)$ is not $S$-equivariant.  Second, we replace $\Gamma \ltimes C_0(X)$ with a twisted crossed product algebra, as in Definition \ref{sec:TwistedCross}.  This will be important later on, when central extensions of the loop group come into the picture.

Consider a semi-direct product $S\ltimes \Gamma$, where $S$, $\Gamma$ are locally compact groups, and assume the $S$ action preserves Haar measure on $\Gamma$.  Let $\Gamma^\tau$ be a $U(1)$-central extension, and assume the action of $S$ on $\Gamma$ lifts to an action on $\Gamma^\tau$, so that we have a $U(1)$-central extension
\[ 1 \rightarrow U(1) \rightarrow S \ltimes \Gamma^\tau \rightarrow S \ltimes \Gamma \rightarrow 1.\]
\ignore{
Let
\[ \kappa \colon \Gamma \rightarrow \Hom(S,U(1)), \qquad \gamma \mapsto \kappa_{\gamma} \in \Hom(S,U(1)) \]
be the group homomorphism determined by the commutator map:
\[ \kappa_{\gamma}(s)=\wh{\gamma} s \wh{\gamma}^{-1} s^{-1} \in U(1),\]
where $\wh{\gamma} \in \Gamma^\tau$ is any lift of $\gamma \in \Gamma$. 
}
The right $(-\tau)$-twisted regular representation $(L^2_{\tau}(\Gamma),\rho)$ (Definition \ref{def:TwistedReg}) extends to a representation of $S \ltimes \Gamma^\tau$ (such that the central circle acts with weight $-1$) according to
\begin{equation} 
\label{eqn:SGammaAction}
\rho(s,\wh{\gamma})f(\hgamma_1)=f(s^{-1}\hgamma_1s\hgamma).
\end{equation}
\ignore{The inverse kappa on the left side I think is necessary, because we are assuming that the central circle acts with weight $-1$ on $L^2_{-\tau}(\Gamma)$, meaning that the commutation relation between $S$ and $\Gamma$ should involve $\kappa_{\gamma}(s)^{-1}$ instead.}The adjoint action $\Ad(\rho)$ on $\bK=\bK(L^2_{\tau}(\Gamma))$ descends to an action of $S \ltimes \Gamma$.  Let $X$ be a locally compact $S\ltimes \Gamma$-space, such that the action of $\Gamma$ on $X$ is proper.  The generalized fixed point algebra $C_0(X \times_\Gamma \bK)$ is an $S$-$C^\ast$ algebra.

\begin{remark}
\label{rem:ExtendSLeft}
For later reference note that the left $\tau$-twisted regular representation $(L^2_\tau(\Gamma),\lambda)$ (Definition \ref{def:TwistedReg}) also extends to a representation of $S\ltimes \Gamma^\tau$ (such that the central circle acts with weight $1$) according to
\[ \lambda(s,\wh{\gamma})f(\hgamma_1)=f(\hgamma^{-1}s^{-1}\hgamma_1s).\]
\end{remark}

If $A$ is a $(S \ltimes \Gamma)$-$C^\ast$ algebra, the twisted crossed product $\Gamma \ltimes_\tau A$ is an $S$-$C^\ast$ algebra, with the $S$ action being the continuous extension of the $S$ action on $C_c(\Gamma^\tau,A)$ given by
\begin{equation} 
\label{eqn:HAction}
(s\cdot a)(\hgamma)= s.a(s^{-1}\hgamma s).
\end{equation}
This applies in particular to the $(S \ltimes \Gamma)$-$C^\ast$ algebra $C_0(X)$, and one has the following variation of Proposition \ref{prop:RieffelFixedPt}.
\begin{proposition}
\label{prop:modRieffelFixedPt}
Consider a semi-direct product $S\ltimes \Gamma$, where $S$, $\Gamma$ are locally compact groups, and the $S$-action preserves Haar measure on $\Gamma$.  Let $\Gamma^\tau$ be a $U(1)$-central extension, and assume the action of $S$ on $\Gamma$ lifts to an action on $\Gamma^\tau$.  Let $(L^2_{\tau}(\Gamma),\rho)$ be the right $(-\tau)$-twisted regular representation (Definition \ref{def:TwistedReg}), extended to a representation of $S\ltimes \Gamma^\tau$ as in equation \eqref{eqn:SGammaAction}, and let $S\ltimes \Gamma^\tau$ act on $\bK=\bK(L^2_{\tau}(\Gamma))$ by the adjoint action $\Ad(\rho)$.  Let $X$ be a locally compact $S \ltimes \Gamma$ space, such that the $\Gamma$ action is proper.  There is an isomorphism of $S$-$C^\ast$ algebras
\[ \Gamma \ltimes_\tau C_0(X) \simeq C_0(X\times_{\Gamma} \bK).\]
\end{proposition}
\begin{proof}
This follows in a straight-forward manner from Proposition \ref{prop:RieffelFixedPt} applied to $\Gamma^\tau$.  The action of $\Gamma$ on $X$ induces a proper action of $\Gamma^\tau$ on $X$ with the central circle acting trivially.  Applying Proposition \ref{prop:RieffelFixedPt} to $\Gamma^\tau$,
\begin{equation} 
\label{eqn:ApplyRieffelFixedPt}
\Gamma^\tau \ltimes C_0(X) \simeq C_0\big(X\times_{\Gamma^\tau} \bK(L^2(\Gamma^\tau))\big)=C_0\big(X\times_\Gamma \bK(L^2(\Gamma^\tau))^{U(1)}\big),
\end{equation}
where for the second equality we use the fact that the central circle acts trivially on $X$.  The algebra on the left hand side of \eqref{eqn:ApplyRieffelFixedPt} splits into a direct sum of its homogeneous ideals
\[ \Gamma^\tau \ltimes C_0(X)=\bigoplus_{n \in \bZ} (\Gamma^\tau \ltimes C_0(X))_{(n)}.\]
Decompose $L^2(\Gamma^\tau)$ into isotypic components for the action of the central circle, as in \eqref{eqn:L2HomogSub}:
\begin{equation} 
\label{eqn:U(1)isotypic}
L^2(\Gamma^\tau)=\bigoplus_{n \in \bZ} L^2(\Gamma^\tau)_{(n)}.
\end{equation}
The subalgebra $\bK(L^2(\Gamma^\tau))^{U(1)} \subset \bK(L^2(\Gamma^\tau))$ is the set of compact operators preserving the decomposition \eqref{eqn:U(1)isotypic}; hence
\begin{equation}
\label{eqn:U(1)compact} 
\bK(L^2(\Gamma^\tau))^{U(1)}=\bigoplus_{n \in \bZ} \bK(L^2(\Gamma^\tau)_{(n)}).
\end{equation}
We claim the isomorphism \eqref{eqn:ApplyRieffelFixedPt} restricts to an isomorphism
\[ (\Gamma^\tau \ltimes C_0(X))_{(n)} \rightarrow C_0\big(X\times_\Gamma \bK(L^2(\Gamma^\tau)_{(n)})\big).\]
To see this let $a \in C_c(\Gamma^\tau \ltimes C_c(X))_{(n)}$, and let $K_a$ be the corresponding family of operators defined by the integral kernels $k_a$ in \eqref{eqn:RieffelIso}.  We suppress the basepoint $x \in X$ from the notation as it plays no role in the argument.  The homogeneity of $a$ (and $U(1)$ invariance of $\mu$) implies (see \eqref{eqn:RieffelIso}) $k_a(\wh{\gamma}_1,z^{-1}\wh{\gamma}_2)=z^{-n}k_a(\wh{\gamma}_1,\wh{\gamma}_2)$, $z \in U(1)$.  For $f \in L^2(\Gamma^\tau)$,
\[ (K_af)(\wh{\gamma}_1)=\int_{\Gamma^\tau} k_a(\wh{\gamma}_1,\wh{\gamma}_2)f(\wh{\gamma}_2)\,d\wh{\gamma}_2.\]
According to \eqref{eqn:AvgU1}, \eqref{eqn:IntGamma} the integral over $\Gamma^\tau$ can be carried out by first averaging with respect to the $U(1)$ action, and then integrating over $\Gamma$.  Note that   
\[ \int_{U(1)} k_a(\wh{\gamma}_1,z^{-1}\wh{\gamma}_2)f(z^{-1}\wh{\gamma}_2)\,dz=k_a(\wh{\gamma}_1,\wh{\gamma}_2)\int_{U(1)} z^{-n} f(z^{-1}\wh{\gamma}_2)\,dz.\]
The integral over $z \in U(1)$ gives the projection to the $(n)$-isotypical component, hence $K_a$ is contained in the ideal $\bK(L^2(\Gamma^\tau)_{(n)})$.  In particular for $n=1$
\[ \Gamma \ltimes_\tau C_0(X) \simeq C_0\big(X\times_\Gamma \bK(L^2(\Gamma^\tau)_{(1)})\big)=C_0\big(X\times_\Gamma \bK(L^2_\tau(\Gamma))\big).\]
\end{proof}
Assuming $\Gamma$ acts on $X$ freely, we can form the associated $S$-equivariant Dixmier-Douady bundle over $X/\Gamma$
\[ \A=X\times_\Gamma \bK,\]
and $\Gamma \ltimes_\tau C_0(X) \simeq C_0(\A)$ as $S$-$C^\ast$ algebras.

\subsection{An example: a Dixmier-Douady bundle $\A_T$ over $T$.}\label{sec:MoritaMorphism}
Let $LG^\tau$ denote a $U(1)$ central extension of the loop group, corresponding to $0 < \ell \in \bZ$ times the generator $LG^\bas$.  Let $T\ltimes \Pi^\tau$ denote the corresponding $U(1)$ central extension of the subgroup $T \times \Pi$ (see Section \ref{sec:loopgroup}).

Carrying out the construction of the previous section with $S=T$, $\Gamma^\tau=\Pi^\tau$, $X=\t$ we obtain a Dixmier-Douady bundle over $T=\t/\Pi$:
\begin{definition}
Let $\A_T$ be the $T$-equivariant associated bundle
\begin{equation}
\label{eqn:defAT} 
\A_T=\t \times_{\Pi} \bK\big(L^2_{\tau}(\Pi)\big) \rightarrow \t/\Pi=T.
\end{equation}
$\A_T$ is a $T$-equivariant Dixmier-Douady bundle over $T$.
\end{definition}

Recall the $G$-equivariant Dixmier-Douady bundle $\A$ described in Section \ref{sec:DDoverG}:
\[ \A=PG \times_{LG} \bK(V^\ast) \rightarrow G \]
where $V$ is a level $\ell$ positive energy representation.  The map
\[ \t \hookrightarrow PG, \qquad \xi \mapsto \gamma_\xi \]
where
\[ \gamma_\xi(s)=\exp(s\xi), \qquad s \in \bR/\bZ \]
embeds $\t$ into $PG$, $\Pi$-equivariantly.  Restricting to $\t \subset PG$ in \eqref{eqn:defA2}, we obtain a Dixmier-Douady bundle
\begin{equation} 
\label{eqn:defA2}
\A|_T=\t \times_\Pi \bK(V^\ast),
\end{equation}
over the maximal torus.  The central circle in $T\ltimes \Pi^\tau$ acts on both $L^2_{\tau}(\Pi)$, $V^\ast$ with weight $-1$ (recall that for $L^2_\tau(\Pi)$ we use the \emph{right} regular representation $\rho$ in Definition \ref{def:TwistedReg}), hence the diagonal action of $T \ltimes \Pi^\tau$ on the tensor product
\[ L^2_{\tau}(\Pi)\otimes V \]
descends to an action of $T\times \Pi$.  Define
\begin{equation} 
\label{eqn:MoritaMorphismT}
\E=\t \times_\Pi \big(L^2_{\tau}(\Pi)\otimes V \big),
\end{equation}
a bundle of Hilbert spaces over $T$.  By \eqref{eqn:defAT} and \eqref{eqn:defA2}, $\E$ defines a $T$-equivariant Morita morphism $\A|_T \dashrightarrow \A_T$.

\subsection{A Green-Julg isomorphism.}\label{sec:GreenJulg}
For a compact group $K$, the Green-Julg theorem states that the $K$-equivariant K-theory of a $K$-$C^\ast$ algebra $A$ is isomorphic to the K-theory of the crossed-product algebra $K\ltimes A$.  There is a `dual' version of the Green-Julg theorem (cf. \cite[Theorem 20.2.7(b)]{Blackadar}) which applies to discrete groups instead of compact groups and K-homology instead of K-theory.
\begin{proposition}
\label{prop:GreenJulg}
Let $\Gamma$ be a discrete group, and let $A$ be a $\Gamma$-$C^\ast$ algebra.  Then
\[ \KK_\Gamma(A,\bC) \simeq \KK(\Gamma \ltimes A,\bC).\]
More generally, suppose a locally compact group $S$ acts on $\Gamma$ preserving Haar measure.  If $A$ is an $S\ltimes \Gamma$-$C^\ast$ algebra then
\[ \KK_{S\ltimes \Gamma}(A,\bC) \simeq \KK_S(\Gamma\ltimes A,\bC),\]
where $\Gamma \ltimes A$ is equipped with the $S$-action in equation \eqref{eqn:HAction}.
\end{proposition}
The isomorphism is simple to describe at the level of cycles.  Let $(H,\pi,F)$ be a cycle representing a class in $\KK(\Gamma \ltimes A,\bC)$.  We may assume $\pi$ is non-degenerate.  The universal property of the crossed product $\Gamma \ltimes A$ guarantees $\pi$ comes from a covariant pair $(\pi_A,\pi_\Gamma)$.  For the triple $(H,\pi_A,F)$ to represent a class in $\KK_\Gamma(A,\bC)$, one needs the operators
\begin{equation} 
\label{eqn:needcompact}
\pi_A(a)(1-F^2), \quad [F,\pi_A(a)], \quad \pi_A(a)(\Ad_{\pi_\Gamma(\gamma)}F-F)
\end{equation}
to be compact, for all $a \in A$, $\gamma \in \Gamma$.  The assumption that $\Gamma$ is discrete means that $A$ is a \emph{sub-algebra} of $\Gamma \ltimes A$, so $\pi_A$ is simply the restriction of $\pi$ to $A$, and the first two operators in \eqref{eqn:needcompact} are compact.  For the last operator, note that 
\begin{equation} 
\label{eqn:IdentityForCommutator}
\pi_A(a)(\Ad_{\pi_\Gamma(\gamma)}F-F)=[F,\pi(a)]-[F,\pi(a)\pi_\Gamma(\gamma)]\pi_\Gamma(\gamma)^{-1}.
\end{equation}
The operator $\pi(a)\pi_\Gamma(\gamma) \in \pi(\Gamma \ltimes A)$, hence the compactness of both terms follows because $(H,\pi,F)$ is a cycle.

The inverse map is similar: a triple $(H,\pi_A,F)$ representing a class in $\KK_\Gamma(A,\bC)$ is sent to the triple $(H,\pi,F)$, where $\pi\colon \Gamma \ltimes A \rightarrow \bB(H)$ is the representation induced by the covariant pair $(\pi_A,\pi_\Gamma)$.  The crossed product $\pi(\Gamma \ltimes A)$ contains a dense sub-algebra consisting of finite linear combinations of operators of the form $\pi(a)\pi_\Gamma(\gamma)=\pi_A(a)\pi_\Gamma(\gamma)$.  The operator $\pi_A(a)\pi_\Gamma(\gamma)(F^2-1)=\pi_\Gamma(\gamma)\pi_A(\gamma^{-1} \cdot a)(F^2-1)$ is compact, while the commutator $[F,\pi_A(a)\pi_\Gamma(\gamma)]$ is compact using \eqref{eqn:IdentityForCommutator} (multiply both sides by $\pi_\Gamma(\gamma)$ on the right).  

The maps are well-defined on homotopy classes because one can apply the same maps to cycles for the pair $(A,C([0,1]))$ (resp. $(\Gamma \ltimes A,C([0,1]))$).

\begin{definition}
Let $N$ be a locally compact group with $U(1)$ central extension $N^\tau$.  Let $A$, $B$ be $N$-$C^\ast$ algebras (trivial $U(1)$ action).  For $n \in \bZ$ define
\[ \KK_{N^\tau}(A,B)_{(n)} \]
to be the direct summand of $\KK_{N^\tau}(A,B)$ generated by cycles where the central circle of $N^\tau$ acts with weight $n$.
\end{definition}
\ignore{
\begin{remark}
One has a restriction homomorphism
\[ \KK_{\widehat{N}}(A,B) \rightarrow \KK_{U(1)}(A,B).\]
Since $U(1)$ acts trivially on $A$, $B$, the latter group is isomorphic to
\[ \KK(A,B)\otimes R(U(1))=\bigoplus_{n \in \bZ} \KK(A,B).\]
The subgroup $\KK_{\widehat{N}}(A,B)^{(1)}$ is the inverse image of the $n=1$ summand.
\end{remark}
}
\begin{proposition}
\label{prop:modGreenJulg}
Consider a semi-direct product $N=S\ltimes \Gamma$, where $S$, $\Gamma$ are locally compact groups and $\Gamma$ is discrete.  Let $\Gamma^\tau$ be a $U(1)$-central extension, and assume the action of $S$ on $\Gamma$ lifts to an action on $\Gamma^\tau$.  Let $A$ be an $S \ltimes \Gamma$-$C^\ast$ algebra.  The twisted crossed product $\Gamma \ltimes_\tau A$ is an $S$-$C^\ast$ algebra with action given by \eqref{eqn:HAction}, and
\[ \KK_{S\ltimes \Gamma^\tau}(A,\bC)_{(1)} \simeq \KK_S(\Gamma \ltimes_\tau A,\bC).\]
\end{proposition}
\begin{proof}
Let $(H,\pi,F,\pi_S)$ represent a class in $\KK_S(\Gamma \ltimes_\tau A,\bC)$.  We may assume $\pi$ is non-degenerate.  The universal property of $\Gamma \ltimes_\tau A$ implies that there is a covariant pair $(\pi_A,\pi_{\Gamma^\tau})$.  Define $\pi_{S\ltimes \Gamma^\tau}(s,\hgamma)=\pi_S(s)\pi_{\Gamma^\tau}(\hgamma)$.  At the level of cycles, the map sends $(H,\pi,F,\pi_S)$ to $(H,\pi_A,F,\pi_{S\ltimes \Gamma^\tau})$.  

We first check that $\pi_{S\ltimes \Gamma^\tau}$ is indeed a representation of $S\ltimes \Gamma^\tau$.  The action of $S$ on $\Gamma \ltimes_\tau A$ extends to an action on the multiplier algebra $M(\Gamma \ltimes_\tau A)$.  By non-degeneracy the representation $\pi$ of $\Gamma \ltimes_\tau A$ extends to $M(\Gamma \ltimes_\tau A)$, and one obtains a covariant pair extending $(\pi,\pi_S)$.  For $\hgamma \in \Gamma^\tau$, the function
\[ u_{\hgamma}(\hgamma^\prime)=\begin{cases} z &\text{ if } \hgamma^\prime=z^{-1}\hgamma, \quad z \in U(1) \\ 0 &\text{ else}\end{cases}\]
lies in $M(\Gamma \ltimes_\tau A)$ and satisfies $\pi(u_{\hgamma})=\pi_\Gamma^\tau(\hgamma)$, $u_{s\hgamma s^{-1}}(\hgamma^\prime)=u_{\hgamma}(s^{-1}\hgamma^\prime s)$.  \ignore{As far as I can remember, the reason for saying multiplier algebra here is because this element lies in $\Gamma \ltimes A^+$ rather than $(\Gamma \ltimes A)^+$, the latter being a slightly smaller algebra (take cross product before or after taking cross product)}By \eqref{eqn:HAction},
\begin{equation} 
\label{eqn:CommutatorGammaS}
\pi_S(s)\pi_\Gamma^\tau(\wh{\gamma})\pi_S(s)^{-1}=\pi_S(s)\pi(u_{\hgamma})\pi_S(s)^{-1}=\pi(s\cdot u_{\hgamma})=\pi(u_{s\hgamma s^{-1}})=\pi_\Gamma^\tau(s\wh{\gamma}s^{-1}).
\end{equation}
Equation \eqref{eqn:CommutatorGammaS} implies that $\pi_{S\ltimes \Gamma^\tau}$ is a representation of $S\ltimes \Gamma^\tau$.

The algebra $A$ can be regarded as a \emph{sub-algebra} of $\Gamma \ltimes_\tau A$, via the embedding $a \mapsto \ti{a}$, where
\[ \ti{a}(\wh{\gamma})=\begin{cases} 
za &\text{ if } \wh{\gamma}=z^{-1}1_{\Gamma^\tau}, \quad z \in U(1) \\ 
0 & \text{ else} \end{cases}\]
and $\pi_A(a)=\pi(\ti{a})$.  The argument that $(H,\pi_A,F)$ represents a class in $\KK_{S\ltimes \Gamma^\tau}(A,\bC)$ is then similar to Proposition \ref{prop:GreenJulg}.  For example, \eqref{eqn:IdentityForCommutator} now reads
\begin{align*}
\pi_A(a)(\Ad_{\pi_{S\ltimes\Gamma^\tau}(s,\wh{\gamma})}F-F)=[F,\pi(\ti{a})]&+\pi(\ti{a})\pi_\Gamma^\tau(\wh{\gamma})(\Ad_{\pi_S(s)}F-F)\pi_\Gamma^\tau(\wh{\gamma})^{-1}\\
&-[F,\pi(\ti{a})\pi_\Gamma^\tau(\wh{\gamma})]\pi_\Gamma^\tau(\wh{\gamma})^{-1}
\end{align*}
(we have used \eqref{eqn:CommutatorGammaS}).  Note $\pi(\ti{a})\pi_\Gamma^\tau(\wh{\gamma}) \in \pi(\Gamma \ltimes_\tau A)$, hence compactness of all three terms follows because $(H,\pi,F)$ is a cycle.

In the reverse direction, let $(H,\pi_A,F,\pi_{S\ltimes \Gamma^\tau})$ represent a class in $\KK_{S\ltimes \Gamma^\tau}(A,\bC)_{(1)}$, and let $\pi_\Gamma^\tau$ (resp. $\pi_S$) be the restriction of $\pi_{S\ltimes \Gamma^\tau}$ to $\Gamma^\tau$ (resp. $S$).  The representations $(\pi_A,\pi_\Gamma^\tau)$ form a covariant pair as in \eqref{eqn:covpair}, and the map sends $(H,\pi_A,F,\pi_{S\ltimes \Gamma^\tau})$ to $(H,\pi,F,\pi_S)$ where $\pi$ is the representation of $\Gamma \ltimes_\tau A$ guaranteed by the universal property.  \ignore{To see why the action of $S$ on $\Gamma \ltimes_\tau A$ is as in \eqref{eqn:HAction}, let $s \in S$, $a \in C_c(\Gamma^\tau,A)$ and $v \in H$, then
\begin{align*} 
\pi_S(s)\pi(a)v&=\int_{\Gamma^\tau} \pi_S(s)\pi_A(a(\wh{\gamma}))\pi_S(s)^{-1}\pi_S(s) \pi_\Gamma^\tau(\hgamma)v \\
&=\int_{\Gamma^\tau} \pi_A(s.a(\wh{\gamma})) \pi_S(s)\pi_{\Gamma}^\tau(\hgamma)v\\
&=\int_{\Gamma^\tau}\pi_A(s.a(\wh{\gamma}))\kappa_\gamma(s)^{-1}\pi_{\Gamma}^\tau(\hgamma)\pi_S(s)v
\end{align*}
and this equals $\pi(s\cdot a)\pi_S(s)v$ if we set $(s\cdot a)(\wh{\gamma})=\kappa_\gamma(s)^{-1}s.a(\wh{\gamma})$.  }One checks that the result is a cycle similar to before.  

The maps are well-defined on homotopy classes because one may apply the same maps to cycles for $(A,C([0,1]))$ (resp. $(\Gamma \ltimes_\tau A,C([0,1]))$). 
\end{proof}

\subsection{The analytic assembly map.}\label{sec:Assembly}
Let $X$ be a locally compact space with a proper action of a locally compact group $N$.  If the action of $N$ is \emph{cocompact}, i.e. $X/N$ is compact, then there is a map
\[ \mu_N\colon \K_0^N(X)=\KK_N(C_0(X),\bC) \rightarrow \KK(\bC,C^\ast(N))=\K_0(C^\ast(N)),\]
known as the \emph{analytic assembly map}.  If $N$ is compact, the analytic assembly map is just the equivariant index:
\[ \mu_N([(H,\rho,F)])=[\ker(F^+)]-[\ker(F^-)] \in \K_0(C^\ast(N))\simeq R(N).\]
For non-compact $N$, the definition of the assembly map is more involved.  We give a brief description here and refer the reader to e.g.  \cite{BaumConnesHigson}, \cite[Section 2]{MislinValette}, \cite[Section 4.2]{EchterhoffBaumConnesReview2017}, \cite{KasparovTransversallyElliptic} for details.

Let $(H,\rho,F)$ be a cycle representing a class $[F] \in \KK_N(C_0(X),\bC)$.  Assume the operator $F$ is \emph{properly supported}, in the sense that for any $f \in C_c(X)$ one can find an $h \in C_c(X)$ such that $\rho(h)F\rho(f)=F\rho(f)$ (this can always be achieved by perturbing $F$, cf. \cite[Section 3]{BaumConnesHigson}).  To define $\mu_N$, the first step is to define a $C_c(N)$-valued inner product $(-,-)_N$ on the subspace $\rho(C_c(X))H \subset H$, by
\[ (f_1,f_2)_N(n)=(f_1,n\cdot f_2)_{L^2}.\]
Complete $\rho(C_c(X))H$ in the norm $\|f\|_N=\|(f,f)_N\|^{1/2}_{C^\ast(N)}$, where $\|-\|_{C^\ast(N)}$ denotes the norm of the $C^\ast$ algebra $C^\ast(N)$, to obtain a Hilbert $C^\ast(N)$-module $\H$.  Then $F$ acts on $\rho(C_c(X))H$ (here use that $F$ is properly supported) and extends to an adjointable operator $\F$ on $\H$.  The pair $(\H,\F)$ represents a class in $\K_0(C^\ast(N))$, and
\[ \mu_N([F])=[(\H,\F)] \in \K_0(C^\ast(N)).\]
Since $\F$ commutes with the $C^\ast(N)$ action, $\ker(\F^{\pm})$ are $C^\ast(N)$-modules, but unfortunately in general they need not be finitely generated and projective, so that `$[\ker(\F^+)]-[\ker(\F^-)]$' is not a K-theory class.  If the range of $\F$ is closed and $\ker(\F^{\pm})$ are finitely generated and projective, then indeed $\mu_N([F])=[\ker(\F^+)]-[\ker(\F^-)]$ (cf. \cite[Proposition 3.27]{HigsonPrimer}); more generally  it is necessary to perturb $\F$ to obtain such a description.

There is another description of the analytic assembly map due to Kasparov that we briefly recall; see for example \cite[Section 4.2]{EchterhoffBaumConnesReview2017} for a recent review, and \cite[Section 2.4]{MislinValette} for a discussion of the relation between the two descriptions of $\mu_N$ (at least for $N$ discrete).  As the action of $N$ on $X$ is cocompact, one can find a continuous compactly supported `cut-off function' $c \colon X \rightarrow [0,\infty)$ such that for all $x \in X$,
\[ \int_N c(n^{-1}.x)^2=1 .\]
Define $p_c \colon G \times X \rightarrow [0,\infty)$ by
\[ p_c(n,x)=\mu(n)^{-1/2}c(n^{-1}.x)c(x),\]
where $\mu$ is the modular homomorphism of $N$.  The function $p_c$ defines a self-adjoint projection in $G\ltimes C_0(X)$, and hence an element $[c] \in \KK(\bC,N\ltimes C_0(X))$.  Kasparov's definition of the assembly map is as a Kasparov product
\[ \mu_N([F])=[c]\otimes_{N\ltimes C_0(X)} j_N([F]),\]
where $j_N \colon \KK_N(C_0(X),\bC) \rightarrow \KK(N\ltimes C_0(X),C^\ast(N))$ is the descent homomorphism.

\section{The group $T\ltimes \Pi^\bas$.}\label{sec:SemiDirect}
In this section we collect results about the group $T\ltimes \Pi^\bas$ and the K-theory of its group $C^\ast$ algebra.  For another discussion of the K-theory of $C^\ast(T\ltimes \Pi^\bas)$ see \cite{Takata1}.  Throughout $G$ is assumed to be connected, simply connected, simple.  Let $T \subset G$ be a fixed maximal torus and we identify $\t \times \t^\ast$ with the basic inner product, and hence $\Pi$ is identified with a sublattice of $\Pi^\ast$.

\subsection{The group $\Pi^\bas$.}
Let $\Pi^\bas$ denote the restriction to $\Pi \subset LG$ of the basic central extension $LG^{\bas}$ of the loop group.  We give an explicit 2-cocycle $\sigma$ for $\Pi^\bas$.  Recall that the cocycle of a $U(1)$ central extension associated to a splitting $\eta \in \Pi \mapsto \heta \in \Pi^\bas$ is the function $\sigma \colon \Pi \times \Pi \rightarrow U(1)$ defined by the equation
\[ \heta_1 \heta_2=\sigma(\eta_1,\eta_2)\wh{\eta_1\eta_2}.\]
(The group operation in $\Pi$ is written multiplicatively.)

Let $\beta_1,...,\beta_r \in \Pi$ be a lattice basis for $\Pi$.  It is known \cite[Proposition 4.8.1]{PressleySegal}, \cite[Theorem 3.2.1]{LaredoPosEnergy} that one can choose lifts $\wh{\beta}_1,...,\wh{\beta}_r \in \Pi^\bas$ such that
\begin{equation} 
\label{eqn:CommutatorLambda}
\wh{\beta}_i \wh{\beta}_j \wh{\beta}_i^{-1}\wh{\beta}_j^{-1}=(-1)^{B(\beta_i,\beta_j)},
\end{equation}
where $B$ is the basic inner product.  For $\eta=\sum n_i \beta_i \in \Pi$ let
\begin{equation} 
\label{eqn:DefOfLift}
\heta=\wh{\beta}_1^{n_1} \cdots \wh{\beta}_r^{n_r}.
\end{equation}
Define a bilinear map
\[ \epsilon \colon \Pi \times \Pi \rightarrow \bZ \]
by
\[ \epsilon(\beta_i,\beta_j)=\begin{cases} B(\beta_i,\beta_j) &\text{ if } i>j\\ 0 &\text{ if } i \le j \end{cases}\]
and extend bilinearly.
\begin{proposition}
The cocycle associated to the splitting \eqref{eqn:DefOfLift} is 
\[ \sigma(\eta_1,\eta_2)=(-1)^{\epsilon(\eta_1,\eta_2)}, \qquad \eta_1,\eta_2 \in \Pi.\]
\end{proposition}
\begin{remark}
The function $(-1)^{\epsilon}$ is the `off-diagonal' part of what Kac \cite[Section 7.8]{KacBook} calls an \emph{asymmetry function}.
\end{remark}
\begin{proof}
If $i>j$ then using \eqref{eqn:CommutatorLambda} we have
\[ \wh{\beta}_i \wh{\beta}_j=(-1)^{B(\beta_i,\beta_j)}\wh{\beta}_j\wh{\beta}_i=(-1)^{B(\beta_i,\beta_j)}\wh{\beta_i\beta_j},\]
while if $i\le j$ then
\[ \wh{\beta}_i\wh{\beta}_j=\wh{\beta_i\beta_j}.\]
This verifies
\[ \sigma(\beta_i,\beta_j)=(-1)^{\epsilon(\beta_i,\beta_j)}\]
for $i,j=1,...,r$.  On the other hand, using the definition of the lift \eqref{eqn:DefOfLift} and the commutation relation \eqref{eqn:CommutatorLambda}, one sees that $\sigma$ is bimultiplicative:
\[ \sigma(\eta_1+\eta_2,\eta)=\sigma(\eta_1,\eta)\sigma(\eta_2,\eta), \qquad \sigma(\eta,\eta_1+\eta_2)=\sigma(\eta,\eta_1)\sigma(\eta,\eta_2).\]
\end{proof}

\subsection{The group $T\ltimes \Pi^\bas$.}
Define a group homomorphism
\begin{equation} 
\label{eqn:defkappa}
\kappa \colon \Pi \rightarrow \Hom(T,U(1))=\Pi^\ast, \qquad \kappa_\eta(t)=t^{-B^\flat(\eta)}.
\end{equation}
It is known (cf. \cite[Section 2.2]{FHTII}, \cite{PressleySegal}) that in the subgroup $T\ltimes \Pi^\bas \subset LG^\bas$, elements $t \in T$ and $\eta \in \Pi^\bas$ satisfy the commutation relation
\[ \heta t \heta^{-1} t^{-1}=\kappa_\eta(t).\]
Moreover the data $(\sigma,\kappa)$ determine the group $T\ltimes \Pi^\bas$ (up to isomorphism).  Let $T\ltimes \Pi^\triv$ denote the analogous group defined by the data $(1,\kappa)$, i.e. $\Pi^\triv=\Pi \times U(1)$ is the trivial central extension, and the commutator map for $T$, $\Pi^\triv$ is the same $\kappa$ defined in \eqref{eqn:defkappa}.
\ignore{I believe this is exactly the same convention as in FHTII, because their $\kappa$ is indeed related to $-B$.  This is also consistent with $-B$ being the restriction of the DD class to $T$, as in Eckhard's conjugacy classes paper.  In our Witten deformation paper, we used the same commutation relation $\eta t \eta^{-1}t^{-1}=t^{-\ell B^\flat(\ol{\eta})}$, except that our $\kappa$ there was exactly $\ell B^\flat$, i.e. our convention for $\kappa$ was opposite, although our formulas for the commutation relation was correct...I believe!}
In detail, if we use the section $\Pi \rightarrow \Pi^\bas$ defined in \eqref{eqn:DefOfLift} to view $\Pi^\bas$ (topologically) as a product $\Pi \times U(1)$, then the group multiplication in $T\ltimes \Pi^\bas$ is
\begin{equation} 
\label{eqn:MultNonTriv}
(t_1,\eta_1,z_1)(t_2,\eta_2,z_2)=(t_1t_2,\eta_1+\eta_2,\kappa_{\eta_1}(t_2)\sigma(\eta_1,\eta_2)z_1z_2)
\end{equation}
while in $T\ltimes \Pi^\triv$ the group multiplication is
\begin{equation} 
\label{eqn:MultTriv}
(t_1,\eta_1,z_1)(t_2,\eta_2,z_2)=(t_1t_2,\eta_1+\eta_2,\kappa_{\eta_1}(t_2)z_1z_2).
\end{equation}

As we saw above, in general $\Pi^\bas$ is not isomorphic to $\Pi^\triv$ ($\Pi^\bas$ need not be abelian).  Perhaps surprisingly, the distinction between $\Pi^\bas$, $\Pi^\triv$ disappears after taking semi-direct product with $T$.
\begin{proposition}
\label{prop:NonCanonicalIso}
The groups $T\ltimes \Pi^\bas$, $T\ltimes \Pi^\triv$ are (non-canonically) isomorphic.
\end{proposition}
\begin{proof}
We will show that the additional sign $\sigma(\eta_1,\eta_2)$ can be absorbed into the phase $\kappa_{\eta_1}(t_2)$, by choosing an appropriate identification $T\ltimes \Pi^\bas \rightarrow T\ltimes \Pi^\triv$.  

For $\eta \in \Pi$, define
\[ \eta_\epsilon=\tfrac{1}{2}B^\sharp(\epsilon(-,\eta)) \in \t,\]
where here one views the contraction $\epsilon(-,\eta)$ as an element of $\t^\ast$, and then uses $B^\sharp$ to convert this to an element of $\t$.  The image of the map $\eta \in \Pi \mapsto \eta_\epsilon \in \t$ is contained in $\tfrac{1}{2}B^\sharp(\Pi^\ast)$.  By construction
\begin{equation}
\label{eqn:AbsorbingProperty}
\exp(\eta_\epsilon)^{B^\flat(\mu)}=e^{\pi \i \epsilon(\mu,\eta)}=\sigma(\mu,\eta), \qquad \eta, \mu \in \Pi.
\end{equation}
Define
\[ \Psi \colon T\ltimes \Pi^\bas \rightarrow T\ltimes \Pi^\triv, \qquad \Psi(t,\eta,z)=(t\exp(\eta_\epsilon),\eta,z). \]
A short calculation using \eqref{eqn:AbsorbingProperty} shows that $\Psi$ is a group homomorphism.
\end{proof}

\subsection{The $C^\ast$ algebra of $T\ltimes \Pi^\bas$.}
Using Proposition \ref{prop:NonCanonicalIso}, $T\ltimes \Pi^\bas \simeq T\ltimes \Pi^\triv$.  There is an obvious isomorphism $(t,\eta,z) \in T\ltimes \Pi^{\triv}\mapsto (t,z,\eta) \in T^\triv \rtimes \Pi$, where $T^\triv=T \times U(1)$ is the trivial central extension.  

If $G_1\ltimes G_2$ is a semi-direct product of locally compact groups, then there is an isomorphism
\[ C^\ast(G_1\ltimes G_2) \simeq G_1 \ltimes C^\ast(G_2) \]
induced by the natural map $C_c(G_1\times G_2) \rightarrow C_c(G_1,C_c(G_2))$, cf. \cite{WilliamsCrossedProducts}.  Thus
\[ C^\ast(T^\triv \rtimes \Pi)\simeq C^\ast(T^\triv) \rtimes \Pi.\]
The group $T^\triv=T\times U(1)$ is abelian, hence $C^\ast(T^\triv)$ is isomorphic to $C_0(\Pi^\ast \times \bZ)$ (the Pontryagin dual).  Thus
\[ C^\ast(T^\triv) \rtimes \Pi \simeq C_0(\Pi^\ast \times \bZ)\rtimes \Pi.\]
If $\xi \in \Pi^\ast$ and $\ell \in \bZ$, the isomorphism $C_0(\Pi^\ast \times \bZ) \rightarrow C^\ast(T^\triv)$ sends $\delta_{(\xi,\ell)}\in C_0(\Pi^\ast \times \bZ)$ to its Fourier transform
\[ e_{\xi,\ell} \in C^\ast(T^\triv), \qquad e_{\xi,\ell}(t,z)=t^\xi z^\ell. \]
Using the commutation relation in $T\ltimes \Pi^{\triv}\simeq T^{\triv}\rtimes \Pi$, the action of $\eta \in \Pi$ on $e_{\xi,\ell}$ is
\[ (\eta \cdot e_{\xi,\ell})(t,z)=e_{\xi,\ell}(t,t^{\eta}z)=t^{\xi+\ell\eta}z^\ell.\]
This corresponds to the action of $\Pi$ on the Pontryagin dual $\Pi^\ast \times \bZ$ by
\begin{equation} 
\label{eqn:LevelNAction}
\eta \cdot (\xi,\ell)=(\xi+\ell\eta,\ell).
\end{equation}
We see that
\[ C_0(\Pi^\ast \times \bZ) \rtimes \Pi=\bigoplus_{\ell \in \bZ} C_0(\Pi^\ast) \rtimes_\ell \Pi \]
where $C_0(\Pi^\ast) \rtimes_\ell \Pi$ denotes the crossed product formed using the `level $\ell$' action \eqref{eqn:LevelNAction}.  The sub-algebra $C^\ast(T\ltimes \Pi^\bas)_{(\ell)}$ corresponds to the $\ell^{th}$ summand.

For $\ell=0$ the action of $\Pi$ on $\Pi^\ast$ is trivial, hence
\[ C_0(\Pi^\ast)\rtimes_0 \Pi \simeq C_0(\Pi^\ast)\otimes C^\ast(\Pi) \simeq C_0(\Pi^\ast \times T^\vee) \]
where $T^\vee=\t^\ast/\Pi^\ast$ is the Pontryagin dual of $\Pi$.  For $\ell \ne 0$, the algebra $C_0(\Pi^\ast)\rtimes_\ell\Pi$ is isomorphic to a direct sum of finitely many copies of the compact operators on $L^2(\Pi)$, indexed by the finite quotient $\Pi^\ast/\ell\Pi$.  One can deduce this from the Takai duality theorem, but it is also not difficult to argue directly as follows.  One has a faithful Schr{\"o}dinger-type representation of $C_0(\Pi^\ast)\rtimes_\ell\Pi$ on $L^2(\Pi^\ast)$, where $C_0(\Pi^\ast)$ acts by multiplication operators, and $\Pi$ acts by translations as in \eqref{eqn:LevelNAction}.  The decomposition of $\Pi^\ast$ into cosets $[\xi]=\xi+\ell\Pi$ for the $\Pi$ action gives a direct sum decomposition
\begin{equation}
\label{eqn:DirectSumLambda} 
L^2(\Pi^\ast)=\bigoplus_{[\xi] \in \Pi^\ast/\ell\Pi} L^2([\xi]),
\end{equation}
and the action of $C_0(\Pi^\ast)\rtimes_\ell\Pi$ preserves this decomposition.  For $\eta \in \Pi$ and $\mu \in \Pi^\ast$ let $\theta_{\eta,\mu} \in C_c(\Pi,C_c(\Pi^\ast))=C_c(\Pi \times \Pi^\ast)$ be the delta function at $(\eta,\mu) \in \Pi \times \Pi^\ast$.  As an element of $C_0(\Pi^\ast)\rtimes_n \Pi$, $\theta_{\eta,\mu}$ acts on $L^2(\Pi^\ast)$ by the rank $1$ linear transformation mapping $\delta_{\mu}$ to $\delta_{\mu+\ell\eta}$.  Such rank $1$ operators generate the algebra of all compact operators on $L^2(\Pi^\ast)$ that preserve the direct sum decomposition \eqref{eqn:DirectSumLambda}, and thus
\[ C_0(\Pi^\ast)\rtimes_\ell \Pi \simeq \bigoplus_{[\xi] \in \Pi^\ast/\ell\Pi} \bK(L^2([\xi])).\]
We summarize these observations with a proposition.
\begin{proposition}
\label{prop:AlgTPi}
The group $C^\ast$ algebra $C^\ast(T\ltimes \Pi^\bas)$ is an infinite direct sum of its homogeneous ideals $C^\ast(T\ltimes \Pi^\bas)_{(\ell)}$, $\ell \in \bZ$.  A choice of group isomorphism $T\ltimes \Pi^\bas \xrightarrow{\sim}T\ltimes \Pi^\triv$ determines isomorphisms of $C^\ast$ algebras
\[ C^\ast(T\ltimes \Pi^\bas)_{(0)}\xrightarrow{\sim} C_0(\Pi^\ast \times T^\vee),\]
and for $\ell \ne 0$
\[ C^\ast(T\ltimes \Pi^\bas)_{(\ell)}\xrightarrow{\sim} \bigoplus_{[\xi] \in \Pi^\ast/\ell\Pi} \bK(L^2([\xi])),\]
where $[\xi]=\xi+\ell\Pi \subset \Pi^\ast$ is a coset for the `level $\ell$' action of $\Pi$ on $\Pi^\ast$.
\end{proposition}

\subsection{The map $\K_0(C^\ast_\tau(T\times \Pi)) \rightarrow R^{-\infty}(T)^{\ell \Pi}$.}\label{sec:KThyTPi}
Let $\tau$ be some integer multiple $0\ne \ell \in \bZ$ of the basic central extension of $LG$, and let $T\ltimes \Pi^\tau$ denote the restriction of $LG^\tau$ to the subgroup $T\times \Pi \subset LG$.  For $\ell=1$ this is precisely the group $T\ltimes \Pi^\bas$ considered above.  Elements $t \in T$ and $\wh{\eta} \in \Pi^\tau$ satisfy the commutation relation
\[ \heta t \heta^{-1} t^{-1}=\kappa_\eta^\ell(t)=t^{-\ell B^\flat(\eta)} \in U(1),\]
see equation \eqref{eqn:defkappa}.

The structure of $C^\ast_\tau(T\times \Pi)$ follows immediately from Proposition \ref{prop:AlgTPi}, and in particular its K-theory is
\[ \K_0(C^\ast_\tau(T\times \Pi))\simeq \bigoplus_{[\xi] \in \Pi^\ast/\ell \Pi} \K_0\big(\bK(L^2([\xi]))\big).\]
The K-theory of $\bK(L^2([\xi]))$ is a copy of the integers, generated by the finitely generated, projective module $L^2([\xi])$.  Let $R^{-\infty}(T)^{\ell\Pi}$ denote the subspace of $R^{-\infty}(T)$ consisting of formal characters invariant under the `level $\ell$' action of $\Pi$, that is, formal sums
\[ \sum_{\xi \in \Pi^\ast} a_\xi e_\xi, \qquad e_\xi(t)=t^\xi \]
where the coefficients satisfy $a_{\xi+\ell \eta}=a_\xi$ for all $\eta \in \Pi$ (we identify $\Pi$ with a sublattice of $\Pi^\ast$ using the basic inner product).  There is a map
\[ \K_0\big(\bK(L^2([\xi]))\big) \rightarrow R^{-\infty}(T)^{\ell \Pi}\]
sending the generator $L^2([\xi])$ to its formal $T$-character:
\[ L^2([\xi]) \mapsto \sum_{\eta \in \Pi} e_{\xi+\ell \eta}.\]
Put differently this formal character has multiplicity function given by the indicator function of the coset $[\xi]$ in $\Pi^\ast$.  It is clear that this map gives an isomorphism of abelian groups:
\begin{equation}
\label{eqn:IsoFormalChar} 
\K_0\big(C^\ast_\tau(T\times \Pi)\big) \xrightarrow{\sim} R^{-\infty}(T)^{\ell \Pi}.
\end{equation}

\section{The map $\scr{I}\colon \K_0^G(G,\A) \rightarrow R^{-\infty}(T)^{W_\aff-\anti,\, \ell}$}\label{sec:defI}
Let $\A$ be a $G$-equivariant Dixmier-Douady bundle over $G$, with Dixmier-Douady class $\ell \in \bZ\simeq H^3_G(G,\bZ)$ and $\ell>0$.  In this section we construct a map
\[ \scr{I}\colon \K_0^G(G,\A) \rightarrow R^{-\infty}(T)^{W_\aff-\anti,\, \ell} \]
and show that in a suitable sense it is an inverse of the Freed-Hopkins-Teleman isomorphism.  We begin by fixing a model for $\A$ as in Section \ref{sec:DDoverG}:
\begin{equation} 
\label{eqn:defA2}
\A=PG \times_{LG} \bK(V^\ast),
\end{equation}
where $V$ is a level $\ell$ positive energy representation of $LG^\bas$.  Let $LG^\tau$ denote the central extension of $LG$ corresponding to $\ell$ times the generator $LG^\bas$, thus $V$ is a representation of $LG^\tau$ such that the central circle acts with weight $1$.  

Let $U$ be a small $N(T)$-invariant tubular neighborhood of $T$ in $G$, with projection map $\pi_T \colon U \rightarrow T$.  A neighborhood $U$ can be described explicitly: for $\epsilon>0$ sufficiently small, and $\st{B}_\epsilon(\t^\perp)$ an $\epsilon$-ball in $\t^\perp \subset \g$, the map
\[ T \times \st{B}_\epsilon(\t^\perp), \qquad (t,\xi)\mapsto t\exp(\xi),\]
is a $N(T)$-equivariant diffeomorphism onto its image, which we may take to be $U$, with $\pi_T$ the projection to the first factor.  The first stage in the definition of $\scr{I}$ is the restriction map
\begin{equation} 
\label{map:restrict}
\K_0^G(G,\A) \rightarrow \K_0^T(U,\A|_U)
\end{equation}
induced by the `extension by $0$' algebra homomorphism $C_0(\A|_U)\hookrightarrow C(\A)$.  

Recall the Dixmier-Douady bundle $\A_T \rightarrow T$ constructed in Section \ref{sec:MoritaMorphism}.  Let $\A_U=\pi_T^\ast \A_T$.  By pullback of \eqref{eqn:MoritaMorphismT} we obtain a Morita equivalence $\A|_U \dashrightarrow \A_U$, and hence also an isomorphism
\begin{equation} 
\label{eqn:MoritaMorph}
\K_0^T(U,\A|_U) \xrightarrow{\sim} \K_0^T(U,\A_U).
\end{equation}
There is a canonical identification $\t^\perp \simeq \g/\t$.  The complexification $(\g/\t)_\bC \simeq \n_+\oplus \n_-$, where $\n_+$ (resp. $\n_-$) is the direct sum of the positive (resp. negative) root spaces.  We choose a complex structure on $\g/\t$ such that $(\g/\t)^{1,0}=\n_-$.  This choice of complex structure determines a Bott-Thom isomorphism
\begin{equation} 
\label{eqn:BottThom}
\K_0^T(U,\A_U) \xrightarrow{\sim} \K_0^T(T,\A_T).
\end{equation}
\ignore{
The Morita morphism \eqref{eqn:MoritaMorph}, composed with the Bott-Thom isomorphism \eqref{eqn:BottThom} gives an isomorphism
\begin{equation}
\label{map:MoritaBott}
\K_0^T(U,\A|_U) \xrightarrow{\sim} \K_0^T(T,\A_T).
\end{equation} }
By equation \eqref{eqn:defAT} and Proposition \ref{prop:modRieffelFixedPt}, the algebra of sections $C_0(\A_T)$ has an alternate description as a twisted crossed product algebra $\Pi \ltimes_\tau C_0(\t)$.  The isomorphism $C_0(\A_T)\xrightarrow{\sim} \Pi\ltimes_{\tau}C_0(\t)$ yields an isomorphism of K-homology groups
\begin{equation}
\label{map:CrossedProd} 
\K_0^T(T,\A_T) \xrightarrow{\sim}\K^0_T(\Pi \ltimes_{\tau} C_0(\t)).
\end{equation}
By Proposition \ref{prop:modGreenJulg} there is a Green-Julg isomorphism
\begin{equation}
\label{map:GreenJulg}
\K^0_T(\Pi \ltimes_{\tau} C_0(\t))\xrightarrow{\sim} \K^0_{T\ltimes \Pi^\tau}(C_0(\t))_{(1)}.
\end{equation}
Since $\Pi$ (hence also $T\ltimes \Pi^\tau$) acts cocompactly on $\t$, we can apply the analytic assembly map:
\begin{equation}
\label{map:assembly}
\K^0_{T\ltimes \Pi^\tau}(C_0(\t)) \rightarrow \K_0(C^\ast(T\ltimes \Pi^\tau)).
\end{equation}
Restricted to $\K^0_{T\ltimes \Pi^\tau}(C_0(\t))_{(1)}$, the image of the assembly map is contained in the direct summand isomorphic to $\K_0(C^\ast_\tau(T\times \Pi))$, and the latter is isomorphic to $R^{-\infty}(T)^{\ell \Pi}$ by \eqref{eqn:IsoFormalChar}.  Composing the maps \eqref{map:restrict}---\eqref{map:assembly} completes the construction of the desired map
\[ \scr{I}\colon \K_0^G(G,\A) \rightarrow R^{-\infty}(T)^{\ell \Pi}.\]
We verify in the next two subsections that the range is the subspace $R^{-\infty}(T)^{W_\aff-\anti, \, \ell}$.
\begin{remark}
The vector space $\t$ is a classifying space for proper actions of $T\ltimes \Pi^\tau$.  The Baum-Connes conjecture says that the assembly map \eqref{map:assembly} is an isomorphism.  The conjecture has been proved for a very large class of groups including, for example, all amenable groups, of which $T\ltimes \Pi^\tau$ is an example (we thank Shintaro Nishikawa for pointing this out).  Consequently, each of the maps in the definition of $\scr{I}$ except the first \eqref{map:restrict} are isomorphisms.
\end{remark}
\begin{remark}
There are slight variations in the order of the maps in the definition of $\scr{I}$ that are equivalent.  For example, let $\U \simeq \t \times \st{B}_\epsilon(\t^\perp)$ be the fibre product $\t \times_T U$, and for $x \in \K_0^G(G,\A)$ let $x_U$ denote the class in $\KK_{T\ltimes \Pi^\tau}(C_0(\U),\bC)_{(1)}$ obtained by applying the composition
\[ \K_0^G(G,\A) \rightarrow \K_0^T(U,\A|_U)\xrightarrow{\sim} \K_0^T(U,\A_U) \xrightarrow{\sim} \KK_T(\Pi\ltimes_\tau C_0(\U),\bC) \xrightarrow{\sim} \KK_{T\ltimes \Pi^\tau}(C_0(\U),\bC)_{(1)} \]
similar to the definition of $\scr{I}$ given above.  Then, identifying $C_0(\U)\simeq C_0(\st{B}_{\epsilon}(\t^\perp))\otimes C_0(\t)$, we have
\begin{equation} 
\label{eqn:DescentFormula}
\scr{I}(x)=\mu_{T\ltimes \Pi^\tau}(\beta \otimes_{C_0(\st{B}_{\epsilon}(\t^\perp))} [x_U])=[c]\otimes_{S \ltimes C_0(\t)} j_S(\beta \otimes_{C_0(\st{B}_{\epsilon}(\t^\perp))} x_U),
\end{equation}
where $\beta \in \K^0_T(\st{B}_{\epsilon}(\t^\perp))$ is the Bott-Thom element, $S=T\ltimes \Pi^\tau$, and for the second equality we use Kasparov's description of the assembly map (Section \ref{sec:Assembly}).
\end{remark}

\subsection{Weyl group symmetry.}
The subgroup $N(T)\subset G$ normalizes $\Pi^\tau$ inside $LG^\tau$.  It follows that there is an action of $N(T)$ by conjugation on $T\ltimes \Pi^\tau$, $L^2_\tau(\Pi)$, and $\Pi \ltimes_\tau C_0(\t) \simeq C_0(\A_T)$.  Hence each of the $C^\ast$ algebras appearing in the definition of $\scr{I}$ is in a natural way an $N(T)$-$C^\ast$ algebra.  There is only one aspect of the definition which is not $N(T)$-equivariant, namely the Bott-Thom map.

Let $N$ be a locally compact group and let $H$ be the connected component of the identity in $N$.  Assume $N$ is unimodular for simplicity.  Let $A$ be an $N$-$C^\ast$ algebra, with $\alpha_A \colon N \rightarrow \tn{Aut}(A)$ the action map.  For $n \in N$ we can view $\alpha_A(n)$ as an isomorphism of $H$-$C^\ast$ algebras $A \rightarrow A^{(n)}$, where $A^{(n)}$ denotes the $C^\ast$ algebra $A$ equipped with the conjugated $H$-action $\alpha_{A^{(n)}}(n^\prime):=\alpha_A(nn^\prime n^{-1})$.  Thus if $A$, $B$ are $N$-$C^\ast$ algebras then any $n \in N$ induces a map
\[ \KK_H(A,B) \rightarrow \KK_H(A^{(n)},B^{(n)}). \]
Composing with the `restriction homomorphism' (\cite[Definition 3.1]{KasparovNovikov}) for the automorphism $\Ad_n \in \Aut(H)$, we obtain an automorphism
\[ \theta_n \colon \KK_H(A,B) \rightarrow \KK_H(A,B). \]
See \cite[Appendix A]{LSQuantLG} for details (note that the notation in \cite[Appendix A]{LSQuantLG} is different).  The automorphism $\theta_n$ acts trivially on elements in the image of the restriction map from $\KK_N(A,B)$, and only depends on the class of the element $n$ in the component group $N/H$.

Let $A$ be an $N$-$C^\ast$ algebra.  A group element $n \in N$ gives rise to an algebra automorphism
\[ \theta_n^A \colon H \ltimes A \rightarrow H \ltimes A,\]
defined on the dense subspace $C_c(H,A)$ by the formula $\theta_n^A(f)(h)=n^{-1}.f(\Ad_nh)$.  In \cite[Appendix A]{LSQuantLG} we show that the corresponding element $\theta_n^A \in \KK(H\ltimes A,H\ltimes A)$ intertwines $\theta_n$ and the descent homomorphism; more precisely
\begin{equation} 
\label{eqn:IntertwiningProperty}
j_H(\theta_n(x))=\theta_n^A \otimes j_H(x) \otimes (\theta_n^B)^{-1},
\end{equation}
for any $x \in \KK_H(A,B)$.

As a special case of the above, consider $H=T \subset N(T)=N$.  As the automorphism $\theta_n$ (resp. $\theta_n^A$, $\theta_n^B$) only depends on 
the class $w=[n] \in N(T)/T=W$, we denote it $\theta_w$ (resp. $\theta_w^A$, $\theta_w^B$).  The Bott-Thom element $\beta \in \K^0_T(\t^\perp)$ is not $N(T)$-equivariant, but instead satisfies (\cite[Proposition 4.8]{LSQuantLG})
\begin{equation} 
\label{eqn:BottAntisymmetry}
\theta_w(\beta)=(-1)^{l(w)}\bC_{\rho-w\rho}\otimes \beta,
\end{equation}
where $\rho$ is the half sum of the positive roots, and $l(w)$ is the length of the Weyl group element $w$.  This is a simple consequence of the fact that (1) $\Ad_n|_{\t^\perp}$ reverses orientation (hence grading) according to the length of $w$, (2) the weight decomposition for $\wedge \n_-$ is not symmetric under the Weyl group.

To simplify notation let $S=T\ltimes \Pi^\tau$.  By \eqref{eqn:DescentFormula} and using an argument similar to that given in  \cite[Section 4.5]{LSQuantLG}, we have
\begin{align*}
\scr{I}(x) \otimes (\theta_w^{\bC})^{-1} &= [c]\otimes j_S(\beta \otimes x_U) \otimes (\theta_w^{\bC})^{-1}\\
&=[c]\otimes (\theta_w^{C_0(\t)})^{-1}\otimes \theta_w^{C_0(\t)}\otimes j_S(\beta \otimes x_U) \otimes (\theta_w^{\bC})^{-1}\\
&=[c] \otimes (\theta_w^{C_0(\t)})^{-1} \otimes j_S(\theta_w(\beta \otimes x_U))\\
&=(-1)^{l(w)}[c]\otimes j_S(\beta \otimes x_U) \otimes \bC_{\rho-w\rho}.
\end{align*}
In third line we used \eqref{eqn:IntertwiningProperty}.  In the fourth line we used \eqref{eqn:BottAntisymmetry}, the $N(T)$-equivariance of $x_U$ (it lies in the image of the restriction map from $\KK_{N(T)\ltimes \Pi^\tau}(C_0(\U),\bC)$), and the fact that the cut-off function $c \colon \t \rightarrow [0,\infty)$ may be chosen to be $N(T)$-invariant, which implies $[c]\otimes (\theta_w^{C_0(\t)})^{-1}=[c]$.  In the last line we are also using that $\K_0(C^\ast(T\ltimes \Pi^\tau))$ is an $R(T)$-module.

\begin{corollary}
The image of $\scr{I}$ is contained in $R^{-\infty}(T)^{W_\aff-\anti,\, \ell}$, the space of formal characters that are alternating under the $\rho$-shifted level $\ell$ action \eqref{eqn:ShiftedAction} of the affine Weyl group.
\end{corollary}

\subsection{Inverse of the Freed-Hopkins-Teleman map.}
The commutative diagram \eqref{diagram:FHT} in the Freed-Hopkins-Teleman theorem implies $\K_0^G(G,\A)$ has a particularly simple $\bZ$-basis obtained by pushforward from $\K_G^0(\pt)$ (together with the Morita morphism $V^\ast \colon \A|_E \dashrightarrow \bC$).  These elements are represented by Kasparov triples $x_\lambda$ with trivial operator $F=0$:
\begin{equation} 
\label{eqn:GeneratorTriple}
x_\lambda=[(V^\ast \otimes R_\lambda, \iota^\ast \otimes \id_{R_\lambda}, 0)]
\end{equation}
Here $R_\lambda \in R(G)$ is the finite-dimensional irreducible representation of $G$ with highest weight $\lambda \in \Pi^\ast_k$, and $\iota^\ast \colon C_0(\A) \rightarrow \A_e \simeq \bK(V^\ast)$ is restriction of a section of $\A$ to the fibre over the identity $e \in G$, so that $\iota^\ast \otimes \id_{R_\lambda}\colon C(\A) \rightarrow \bB(V^\ast \otimes R_\lambda)$ is a representation of $C(\A)$ on the Hilbert space $V^\ast \otimes R_\lambda$, with range contained in the compact operators.  By \eqref{diagram:FHT} the corresponding element of $R_k(G)$ is the image $[R_\lambda] \in R_k(G)$ of $R_\lambda \in R(G)$ under the quotient map.  Under the isomorphism \eqref{eqn:AlternatingFormal}, $[R_\lambda]$ is sent to the formal character
\begin{equation} 
\label{eqn:ImgFHT}
\sum_{w \in W_{\aff}} (-1)^{l(w)} e_{w\bullet_{\ell} \lambda} \in R^{-\infty}(T).
\end{equation}

It is easy to determine $\scr{I}(x_\lambda)$.  Let $R_\lambda^T$ denote the $\bZ_2$-graded representation of $T$ corresponding to the numerator of the Weyl character formula for $R_\lambda$, thus $R_\lambda^T$ has character
\[ \sum_{\ol{w} \in W} (-1)^{l(\ol{w})}e_{\ol{w}(\lambda+\rho)-\rho}. \]
By the Weyl character formula the characters $\chi(R_\lambda|_T)$, $\chi(R_\lambda^T)$ are related by
\[ \chi(R_\lambda^T)=\chi(R_\lambda|_T)\cdot \chi(\wedge \n_-) \]
where $\wedge \n_-$ denotes the $\bZ_2$-graded representation of $T$ with character
\[ \prod_{\alpha \in \R_-} (1-e_\alpha).\]
In defining the Bott-Thom map we used a complex structure on $\g/\t$ such that $(\g/\t)^{1,0}=\n_-$.  It follows that the image of $x_\lambda$ under restriction to $U \subset G$, followed by the Bott-Thom map is
\begin{equation} 
\label{eqn:ImgGenBott}
[(V^\ast\otimes R_\lambda^T,\iota^\ast \otimes \id_{R_\lambda^T},0)].
\end{equation}
Applying the Morita morphism $\A|_T \dashrightarrow \A_T$ to \eqref{eqn:ImgGenBott} swaps $L^2_{\tau}(\Pi)$ for $V^\ast$.  The Green-Julg map followed by the assembly map send this element to the class of the $C^\ast_\tau(T\times \Pi)$-module
\begin{equation} 
\label{eqn:ImgAssembly}
L^2_{\tau}(\Pi)\otimes R_\lambda^T 
\end{equation}
in $\K_0(C^\ast_\tau(T\times \Pi))$, where $T \ltimes \Pi^\tau$ acts on $L^2_{\tau}(\Pi)$ (see Remarks \ref{rem:LeftReg}, \ref{rem:ExtendSLeft}) by
\[ (t,\heta)\cdot f(\heta^\prime)=\kappa_{\eta^\prime}(t)^{-1}f(\heta^{-1}\heta^\prime)=t^{\ell \eta^\prime}f(\heta^{-1}\heta^\prime).\]
Since the formal character of \eqref{eqn:ImgAssembly} is exactly \eqref{eqn:ImgFHT}, we have proven the following.
\begin{proposition}
\label{prop:InverseFHT}
Let $k>0$ and let $\A$ be a Dixmier-Douady bundle over $G$ with Dixmier-Douady class $\ell=k+\hvee \in \bZ \simeq H^3_G(G,\bZ)$.  The isomorphism
\[ R_k(G) \simeq R^{-\infty}(T)^{W_\aff-\anti,\,\ell}\] 
intertwines $\scr{I}$ with the inverse of the Freed-Hopkins-Teleman isomorphism.
\end{proposition}
\begin{remark}
Without using the Freed-Hopkins-Teleman theorem, the arguments above show that the map $\scr{I}\colon \K_0^G(G,\A) \rightarrow R^{-\infty}(T)^{W_\aff-\anti, \, \ell}$ is at least surjective.
\end{remark}

\section{Specialization to geometric cycles}\label{sec:IndexMap}
Throughout this section let $\A$ be a $G$-equivariant Dixmier-Douady bundle over $G$ with $\tn{DD}(\A)=\ell=k+\hvee \in \bZ \simeq H^3_G(G,\bZ)$, with $k>0$.  Let $(M,E,\Phi,\S)$ be a D-cycle representing the class $x=(\Phi,\S)_\ast[\scr{D}^E] \in \K^G_0(G,\A)$.  In this section we exhibit $\scr{I}(x)$ as the $T$-equivariant $L^2$-index of a $1^{st}$-order elliptic operator on a non-compact manifold.

\subsection{A cycle for the K-homology push-forward.}\label{sec:Pushforward}
As a first step we describe an analytic cycle representing $x \in \K^G_0(G,\A)$.  To put this in context, one should compare the standard example \ref{ex:DeRhamDirac}.  The result will be a cycle given in terms of a `Dirac operator' acting on sections of a Clifford module, except that the module will have infinite rank (since $\S$ has infinite rank).  The action of the $C^\ast$ algebra $C(\A)$ plays an essential role in making the result a well-defined analytic cycle.  The construction works more generally, with the target space $G$ replaced by any compact Riemannian $G$-manifold $X$.

The push-forward $(\Phi,\S)_\ast [\scr{D}^E]$ is given by the $\KK$-product $[\S]\otimes [\scr{D}^E]$, see \eqref{eqn:PushNotation} and \eqref{eqn:ProdNotation}.  The Hilbert space of the $\KK$-product is described by:
\begin{proposition}
\label{prop:HilbertSpaceIso1}
There is an isomorphism
\[ C_0(\S)\wh{\otimes}_{\Cl(M)} L^2(M,\Cliff(TM)\otimes E) \simeq L^2(M,\S\otimes E)\]
of $\bZ_2$-graded representations of $C_0(\A)$.
\end{proposition}
\noindent The proof is essentially the same as for the standard example \ref{ex:DeRhamDirac}.  For the reader's benefit we include a proof in the appendix.

Recall that $\S$ is a right $\Cliff(TM)$-module, and let
\begin{equation} 
\label{eqn:CliffActionS}
\c \colon \Cliff(TM) \rightarrow \End(\S) 
\end{equation}
denote the action.  Let
\begin{equation} 
\label{eqn:TwistedCliffActionS}
\hc \colon \Cliff(TM) \rightarrow \End(\S), \qquad \hc(v)s=(-1)^{\deg(s)} \c(v)s, \quad v \in TM
\end{equation}
denote the action with a `twist' coming from the grading.  Choose $G$-invariant Hermitian connections $\nabla^E$ and $\nabla^{\S}$ on $\S$, and let $\nabla^{\S \otimes E}$ denote the induced connection on $\S \otimes E$.  Assume moreover that $\nabla^{\S}$ is chosen satisfying
\begin{equation} 
\label{eqn:CliffConnection}
\nabla^{\S}_v (\c(\varphi)s)=\c(\nabla_v\varphi)s+\c(\varphi)\nabla^{\S}_v s,
\end{equation}
i.e. $\nabla^{\S}$ is a \emph{Clifford connection} (cf. \cite[Definition 3.39]{BerlineGetzlerVergne}).  Such a connection can be constructed as in the case of a finite dimensional Clifford module.  In short, one constructs the connection locally and then patches the local definitions together with a partition of unity.  Locally on $U \subset M$ one can find a spin structure $S^{\tn{spin}}$, and $\S|_U\simeq S^{\tn{spin}} \otimes \S^\prime$ as $\Cliff(TM)$-modules, with $\S^\prime=\Hom_{\Cliff(TM)}(S^{\tn{spin}},\S|_U)$.  Using the spin connection on $S^{\tn{spin}}$ and any Hermitian connection on $\S^\prime$ produces a Clifford connection on $\S|_U$.

The candidate Dirac-type operator $\st{D}^E$ acting on smooth sections of $\S \otimes E$ is the composition
\begin{equation} 
\label{eqn:DefDE}
\Gamma^\infty(\S\otimes E) \xrightarrow{\nabla^{\S\otimes E}} \Gamma^\infty(T^\ast M \otimes \S \otimes E) \xrightarrow{g^\sharp} \Gamma^\infty(TM\otimes \S \otimes E) \xrightarrow{\hc} \Gamma^\infty(\S\otimes E). 
\end{equation}

\begin{proposition}
\label{prop:TwistedKHomCyc}
The operator $\st{D}^E$ defined in \eqref{eqn:DefDE} is essentially self-adjoint.  The triple $(L^2(M,\S\otimes E),\rho,\st{D}^E)$ is an unbounded cycle for an element of $\K^G_0(X,\A)$.
\end{proposition} 
\begin{proof}
The presence of a vector bundle $E$ does not alter the proof, so we set $E=\bC$ to simplify notation.  The condition that $\nabla^{\S}$ is a Clifford connection ensures $\st{D}$ is symmetric, as for a finite dimensional Clifford module (cf. \cite[Proposition 5.3]{LawsonMichelsohn}).  It is possible to extend certain proofs of the essential self-adjointness of a Dirac operator on a finite dimensional vector bundle over a compact manifold quite directly to the case of a smooth Hilbert bundle, cf. \cite[Proposition 1.16]{Ebert2016Index} for details.

It suffices to check that for a dense set of $a \in \Gamma^\infty(\A)$, (1) the commutator $[\st{D},\rho(a)]$ is bounded, and (2) the operator $\rho(a)(1+\st{D}^2)^{-1}$ is compact.  Since the underlying space $X$ is compact, we can find a finite open cover such that for each $U$ in the cover, $\A|_{U} \simeq U \times \bK(H)$ for some Hilbert space $H$, $\S|_{U} \simeq U \times (H \otimes F)$ with $F$ a finite dimensional vector space, and the action $\rho$ of $\A|_{U}$ on $\S|_{U}$ is given by the defining representation of $\bK(H)$ on the first factor in $H \otimes F$.  Using a partition of unity subordinate to the cover, we can assume $a$ has support contained in a single $U$, and moreover that $a$ is of the form $a=fb$ where $f \in C^\infty_c(U)$ and $b \in \bK(H)$ is a constant operator.  For the first assertion, note that
\[ [\st{D},\rho(fb)]=\hc(g^\sharp(df))\rho(b)+f[\st{D},\rho(b)].\]
The first term is bounded since $f$ is smooth. The second term is bounded because on $U$, $\st{D}=\st{D}_0+A$, where $\st{D}_0$ is defined in the same way as $\st{D}$ but using the trivial connection on $U$ (hence $[\st{D}_0,\rho(b)]=0$), and $A$ is a bounded bundle endomorphism.

For the second assertion, it is convenient to assume that $b$ also has finite constant rank.  The range of the operator $(1+\st{D}^2)^{-1}$ is contained in the Sobolev space $\H^2(M,\S)$ of sections with two derivatives in $L^2$, hence the range of $\rho(a)(1+\st{D}^2)^{-1}$ is contained in the space $f \cdot \H^2(U,\ran(b)\otimes F)$.  It follows that the operator $\rho(a)(1+\st{D}^2)^{-1}$ factors through the inclusion
\[ f\cdot \H^2(U,\ran(b)\otimes F) \hookrightarrow L^2(U,\ran(b)\otimes F).\]
Since $\ran(b)\otimes F$ is finite dimensional, the Rellich Lemma implies this inclusion is compact.\ignore{This almost seems a little too slick.  But I think the argument is correct.  The range of the operator is contained in the space of $L^2$ sections of $\ran(b)\otimes V$ to which we can apply $\st{D}$ twice and still land in $L^2$, i.e. in the ordinary Sobolev space $\H^2(U_i,\ran(b)\otimes V)$!}
\end{proof} 

\begin{theorem}
The cycle $(L^2(M,\S\otimes E),\rho,\st{D}^E)$ represents the class $[\S]\otimes [\scr{D}^E] \in \K^G_0(X,\A)$.
\end{theorem}
The proof is essentially the same as the standard example \ref{ex:DeRhamDirac}, see the appendix.

\subsection{The Morita morphism $\A|_U \dashrightarrow \A_U$.}
Recall from Section \ref{sec:defI} that $U$ denotes an $N(T)$-invariant tubular neighborhood of $T$ in $G$, and $\pi_T \colon U \rightarrow T$ the projection map.  Let 
\[ Y=\Phi^{-1}(U) \subset M, \qquad \Phi_T=\pi_T \circ \Phi.\]
The restriction of the Morita morphism $\Cliff(TM) \dashrightarrow \Phi^\ast \A$ to $Y$ is a morphism $\Cliff(TY) \dashrightarrow \Phi^\ast \A|_U$.  Composing with the morphism $\A|_U \dashrightarrow \A_U$ of Section \ref{sec:defI} gives a Morita morphism
\begin{equation} 
\label{eqn:YAT}
\V \colon \Cliff(TY) \dashrightarrow \Phi^\ast \A_U.
\end{equation}

The pullback $\A_{\t}=\exp^\ast \A_T=\t \times \bK(L^2_\tau(\Pi))$ has a canonical $T\ltimes \Pi^\tau$-equivariant Morita trivialization $\A_{\t} \dashrightarrow \underline{\bC}$ given by the $\A_{\t}^{\op}$-module $\t \times L^2_\tau(\Pi)^\ast$.  Hence, we have a pullback diagram
\[ \begin{CD}
\Y @>\Phi_{\t}>> \t\\
@Vq_Y VV	@VV\exp V\\
Y @>\Phi_T >> T
\end{CD}\]
and the pullback of $\V$ to $\Y$ is a Morita morphism
\begin{equation}
\label{PullbackV}
q_Y^\ast \V \colon \Cliff(q_Y^\ast TY)\simeq \Cliff(T\Y) \dashrightarrow \Phi_{\t}^\ast \A_{\t}.
\end{equation}
Composing \eqref{PullbackV} with the Morita trivialization of $\A_{\t}$, we obtain a $T\ltimes \Pi^\tau$-equivariant Morita trivialization
\[ S \colon \Cliff(T\Y) \dashrightarrow \underline{\bC}, \]
or in other words, a $T\ltimes \Pi^\tau$-equivariant spinor module for $\Cliff(T\Y)$.  Thus $S$ is a finite dimensional $T\ltimes \Pi^\tau$-equivariant $\bZ_2$-graded Hermitian vector bundle over $\Y$, together with an isomorphism $\c \colon \Cliff(T\Y)\xrightarrow{\sim} \End(S)$.

The central circle in $\Pi^\tau$ acts on $L^2_\tau(\Pi)$, $S$ with opposite weight (for the action on $L^2_\tau(\Pi)$ we use the right regular representation, for which the weight of the central circle action is $-1$), and hence the diagonal $\Pi^\tau$ action on $L^2_\tau(\Pi) \otimes S$ descends to an action of $\Pi$.  By construction, the $\Phi^\ast \A_U$-$\Cliff(TY)$ bimodule $\V$ is the quotient 
\begin{equation} 
\label{eqn:RepresentV}
\V=(L^2_\tau(\Pi) \otimes S)/\Pi. 
\end{equation}
Let $[\V] \in \KK_T(C_0(\A_U),\Cl(Y))$ denote the corresponding $\KK$-element defined by the pair $(\Phi|_Y,\V)$.  The action of $C_0(\A_U)$ on the right hand side in \eqref{eqn:RepresentV} is as follows.  Given $a \in C_0(\A_U)$, the pullback $q_Y^\ast \Phi^\ast a$ is a $\Pi$-invariant map $\Y \rightarrow \bK(L^2(\Pi))$, hence acts on the first factor of $L^2(\Pi)\otimes S$ by the defining representation for $\bK(L^2_\tau(\Pi))$.  This action preserves the space of $\Pi$-invariant sections of $L^2_\tau(\Pi) \otimes S$, hence descends to an action $\rho$ of $C_0(\A_U)$ on $C_0(\V)$.  The action of $\Cl(Y)$ on the right hand side in \eqref{eqn:RepresentV} can be described in similar terms.

The restriction of the fundamental class $[\scr{D}]$ of $M$ to $Y$ is the fundamental class of $Y$, and we will abuse notation slightly denote it by $[\scr{D}]$ also.  By functoriality of the Kasparov product, the image of $(\Phi,\S)_\ast[\scr{D}^E]|_U$ under the Morita morphism $\A|_U \dashrightarrow \A_U$ equals the $\KK$-product
\[ [\V] \otimes [\scr{D}^E] \in \KK_T(C_0(\A_U),\bC).\]

\subsection{The Dirac operator on $\Y$.}
Choose a complete $N(T)$-invariant Riemannian metric on $Y$.  The Kasparov product $[\V]\otimes [\scr{D}^E] \in \KK_T(C_0(\A_U),\bC)$ is represented by a cycle $(H,\rho,\st{D}^E)$ similar to Section \ref{sec:Pushforward}, with now $H=L^2(Y,\V\otimes E)$.  This cycle has an alternate interpretation as the class represented by a Dirac operator on the covering space $\Y$.  The correspondence between differential operators on $Y$ and $\Y$ that we make use of is well-known, cf. \cite[Section 7.5]{SchickL2}, \cite{AtiyahL2,SingerL2} for further details.
\begin{proposition}
\label{prop:IsoWithCoveringSpace}
There is a $N(T)$-equivariant isomorphism of Hilbert spaces
\[ L^2(Y,\V\otimes E) \simeq L^2(\Y,S\otimes E),\]
intertwining the Clifford actions and preserving the subspaces of smooth compactly supported sections.  Under this isomorphism the operator $\st{D}^E$ in $L^2(Y,\V\otimes E)$ corresponds to the Dirac operator in $L^2(\Y,S\otimes E)$.
\end{proposition}
\begin{proof}
Let $s \in C^\infty_c(\Y,S)$ be a smooth compactly supported section of $S$, and let $\delta \in L^2_\tau(\Pi)$ denote the function
\[ \delta(\hgamma)=\begin{cases} z &\text{ if } \hgamma=z^{-1}1_{\Gamma^\tau}\\ 0 &\text{ else.}\end{cases} \]
(This element plays the role of the delta function of $L^2(\Pi)$ supported at $1_{\Pi}$.)  Define a smooth section $\ti{s}$ of the bundle of Hilbert spaces $L^2_\tau(\Pi) \otimes S$ over $\Y$ by `averaging over $\Pi$':
\[ \ti{s}(y)=\sum_{\eta \in \Pi} \eta.\big(\delta \otimes s(\eta^{-1}.y)\big) \]
where here we use the fact that $\Pi$ acts on $L^2_\tau(\Pi)\otimes S$ (the summand on the right could also be written $\heta.\delta \otimes \heta.s(\eta^{-1}.y)$, for any lift $\heta \in \Pi^\tau$ of $\eta$).  The section $\ti{s}$ is $\Pi$-invariant, hence descends to a section of $\V$, which is again smooth and compactly supported.  The map intertwines the $L^2$ norms, hence extends to a unitary mapping. It's clear that the map intertwines the Clifford actions, and hence also the corresponding Dirac operators.
\end{proof}

Abusing notation slightly, we continue to write $\st{D}^E$ (resp. $\rho$) for the Dirac operator on the covering space $\Y$ acting on sections of $S\otimes E$ (resp. the representation of $C_0(\A_U)$ on $L^2(\Y,S\otimes E)$ induced by the isomorphism in Proposition \ref{prop:IsoWithCoveringSpace}).
\begin{corollary}
The product $[\V]\otimes [\scr{D}^E]$ is the class $[\st{D}^E]$ represented by the triple 
\[(L^2(\Y,S\otimes E),\rho,\st{D}^E).\]
\end{corollary}

\subsection{The Bott-Thom map.}\label{sec:BottThom}
Recall we chose a complex structure on $\t^\perp$ such that $(\t^\perp)^{1,0}=\n_-$; thus the complex weights of the $T$-action on $\t^\perp$ in the adjoint representation are the negative roots.  The Bott-Thom class $[\beta \in \K^0_T(\t^\perp)$ is represented by the triple $(C_0(\t^\perp)\otimes \wedge \n_-,\rho,\beta)$, where $\beta \colon \t^\perp \rightarrow \End(\wedge \n_-)$ is the bundle endomorphism given at $\xi \in \t^\perp$ by the Clifford action of $\xi$ on the spinor module $\wedge \n_-$ for $\Cl(\t^\perp)$.  

Choose a diffeomorphism $\st{B}_\epsilon(\t^\perp) \xrightarrow{\sim} \t^\perp$ which we use to pull the Bott element back to an element of $\K^0_T(\st{B}_\epsilon(\t^\perp))$.  Taking the external product with the identity element in $\KK_T(C(\A_T),C(\A_T))$ and using the isomorphism 
\[ C_0(\A_U)\simeq C_0(\st{B}_\epsilon(\t^\perp))\otimes C(\A_T)\]
we obtain an invertible element, still denoted $[\beta]$, in the group
\[ \KK_T(C(\A_T),C_0(\A_U)). \]
The Bott-Thom isomorphism $\KK_T(C_0(\A_U),\bC) \xrightarrow{\sim} \KK_T(C(\A_T),\bC)$ is given by Kasparov product with this element.

The next step is to describe a cycle representing the product 
\[ [\beta] \otimes [\st{D}^E] \in \KK_T(C(\A_T),\bC).\]
We studied a similar product in \cite[Section 4.7]{LSQuantLG}, and we simply state the result.  The operator $\st{D}^E$ is extended to sections of $\wedge \n_- \wh{\otimes} S \otimes E$ (we use the same symbol for the extension) such that
\[ \st{D}^E( \alpha \wh{\otimes} \sigma)=(-1)^{\tn{deg}(\alpha)}\alpha \wh{\otimes}\st{D}^E \sigma \]
whenever $\alpha \in \wedge^{\tn{deg}(\alpha)} \n_-$ is constant and $\sigma$ is a section of $S\wh{\otimes}E$.  The product is represented by the triple 
\[ (L^2(\Y,\wedge \n_- \wh{\otimes} S \otimes E),\rho \circ \pi_T^\ast, \st{D}^E_\beta), \qquad \st{D}^E_\beta = \st{D}^E+\beta_\Y\]
where $\beta_\Y$ is the pullback, via the map 
\[\Y \xrightarrow{\pi} Y \xrightarrow{\Phi} U \simeq T \times \st{B}_\epsilon(\t^\perp) \simeq T \times \t^\perp \xrightarrow{\pr_2} \t^\perp, \] 
of the odd bundle endomorphism $\beta \colon \t^\perp \rightarrow \End(\wedge \n_-)$ described above.

\subsection{The analytic assembly map and the index.}
In \cite[Section 4.7]{LSQuantLG} we verified that the operator $\st{D}^E_\beta=\st{D}^E+\beta_\Y$ is $T$-Fredholm, i.e. the multiplicity of each irreducible representation of $T$ in the $L^2$-kernel $\ker(\st{D}^E_\beta)$ is finite.  Thus $\st{D}^E_\beta$ has a well-defined `$T$-index' denoted $\index(\st{D}^E_\beta)\in R^{-\infty}(T)$, see \cite[Section 2.5]{LSQuantLG}.

Via the isomorphism
\[ \KK_T(C(\A_T),\bC)\simeq \KK_{T\ltimes \Pi^\tau}(C_0(\t),\bC)_{(1)} \]
the element $[\st{D}^E_\beta]$ is identified with an element $[\st{D}^E_\beta] \in \KK_{T\ltimes \Pi^\tau}(C_0(\t),\bC)_{(1)}$.

\begin{proposition}
The image of the class $[\st{D}^E_\beta]$ under the composition
\[ \KK_{T\ltimes \Pi^\tau}(C_0(\t),\bC)_{(1)} \xrightarrow{\mu_{T\ltimes \Pi^\tau}} \KK(\bC,C^\ast_\tau(T\times \Pi)) \simeq R^{-\infty}(T)^{\ell\Pi} \]
is the formal character $\index(\st{D}^E_\beta)$.
\end{proposition}
\begin{proof}
Let $N=T\ltimes \Pi^\tau$.  Let $H=L^2(\Y,\wedge \n_- \wh{\otimes} S \otimes E)$ and let $\H$ be the Hilbert $C^\ast(N)$-module obtained as the completion of $C_c(\t)H$ with respect to the norm defined by the $C^\ast(N)$-valued inner product
\[ (s_1,s_2)_{C^\ast(N)}(n)=(s_1,n\cdot s_2)_{L^2} \]
as in Section \ref{sec:Assembly}.  This inner product takes values in the ideal $C^\ast(N)_{(1)} \subset C^\ast(N)$.  Let $\chi \colon \bR \rightarrow [-1,1]$ be a smooth \emph{normalizing function}, that is, $\chi$ is an odd function, $\chi(t)>0$ for $t >0$ and $\lim_{t \rightarrow \pm \infty}\chi(t)=\pm 1$.  We can moreover choose $\chi$ to have compactly supported Fourier transform.  The operator $F=\chi(\st{D}^E_\beta)$ is then a bounded, properly supported operator on $H$, with the same $T$-index as $\st{D}^E_\beta$, see \cite[Chapter 10]{HigsonRoe}.  $F$ preserves the subspace $C_c(\t)H$, and its restriction extends to a bounded operator $\F$ on $\H$.  The image of $[\st{D}^E_\beta]$ under the analytic assembly map $\mu_N$ is the class in $\K_0(C^\ast(N)_{(1)})$ represented by the pair $(\H,\F)$.

Recall that the ideal $C^\ast(N)_{(1)}$ is isomorphic to a finite direct sum of copies of the compact operators on $L^2(\Pi)$:
\begin{equation} 
\label{eqn:BlockDiag}
C^\ast(N)_{(1)} \simeq \bigoplus_{[\xi] \in \Pi^\ast/\ell \Pi} \bK(L^2([\xi])),
\end{equation}
where $[\xi] \subset \Pi^\ast$ is viewed as a coset of the action of $\ell \Pi$ on $\Pi^\ast$.  There is in particular a faithful representation
\[\rho \colon C^\ast(N)_{(1)} \rightarrow \bK(L^2(\Pi^\ast)) \]
with image the block diagonal subalgebra \eqref{eqn:BlockDiag} of $\bK(L^2(\Pi^\ast))$.  For $s_1,s_2 \in C_c(\t)H$, a short calculation shows that
\begin{equation} 
\label{eqn:TraceNorm}
\Tr(\rho(f))=(s_1,s_2)_{L^2}, \qquad f=(s_1,s_2)_{C^\ast(N)}.
\end{equation}
The norm of an element $f \in C^\ast(N)_{(1)}$ is equal to the operator norm of $\rho(f)$.  Thus for $s \in C_c(\t)H$, its norm in $\H$ is $\|\rho(f)\|^{1/2}$, where $f=(s,s)_{C^\ast(N)}$.  Using \eqref{eqn:TraceNorm} and since $f$ is a positive element, one has $\|\rho(f)\| \le \Tr(\rho(f))=\|s\|^2_{L^2}$.  It follows that $H \hookrightarrow \H$, and corresponds to the subspace of $s \in \H$ such that $\rho(f)$ is trace class, where $f=(s,s)_{C^\ast(N)}$.

The Hilbert $C^\ast(N)_{(1)}$-module $\H$ splits into a finite direct sum:
\[ \H=\bigoplus_{[\xi]\in \Pi^\ast/\ell \Pi} \H_{[\xi]}, \qquad \H_{[\xi]}=\ol{\H\cdot \bK(L^2([\xi]))} \]
with $\H_{[\xi]}$ a Hilbert $\bK(L^2([\xi]))$-module.  The operator $\F$ commutes with the $C^\ast(N)_{(1)}$ action, hence preserves this decomposition, and induces a generalized Fredholm operator $\F_{[\xi]}$ on each $\H_{[\xi]}$.  By the strong Morita equivalence $\bK(L^2([\xi])) \sim \bC$, any countably generated Hilbert $\bK(L^2([\xi]))$-module can be realized as a direct summand of $\bK(V)$, for some $V$.  The generalized Fredholm operator $\F_{[\xi]}$ can be extended by the identity to $\bK(V)$, giving a generalized Fredholm operator $\F_V$ on $\bK(V)$.

Let $V$ be an infinite dimensional Hilbert space and $\bK(V)$ the compact operators.  When $\bK(V)$ is viewed as a right Hilbert $\bK(V)$-module, the space of (bounded) adjointable operators is naturally identified with $\bB(V)$ acting by left multiplication, while the space of generalized compact operators is $\bK(V) \subset \bB(V)$ \cite{WeggeOlsen}.  Thus the generalized Fredholm operators, in the sense of Hilbert modules, on $\bK(V)$, are precisely the operators given by left multiplication by a Fredholm operator on $V$ in the ordinary sense.  It follows from Atkinson's theorem that a generalized Fredholm operator $\F_V$ on $\bK(V)$ has closed range.  If $\F_V$ is left multiplication by $F_V \in \bB(V)$ then $\ran(\F)=\bK(V,\ran(F_V))$ while $\ker(\F)=\bK(V,\ker(F_V))$.  As $\ker(F_V)$ is finite-dimensional, $\bK(V,\ker(F_V))\simeq V \otimes \ker(F_V)$ is a finitely generated, projective $\bK(V)$-module, and also a Hilbert space; moreover, the Hilbert space inner product is given by the composition of the $\bK(V)$-valued inner product with the trace.

By the above generalities, the generalized Fredholm operator $\F_{[\xi]}$ on $\H_{[\xi]}$ must have closed range, and hence the same is true for $\F$.  Moreover
\[ \mu_N([\st{D}^E_{\beta}])=[\ker(\F^+)]-[\ker(\F^-)] \in \K_0(C^\ast(N)_{(1)}), \]
with $\ker(\F^{\pm})$ being Hilbert spaces, with the inner product given by the composition of the $\bK(L^2(\Pi^\ast))$-valued inner product with the trace.  But the latter agrees with the $L^2$-inner product in $H$ by \eqref{eqn:TraceNorm}, hence $\ker(\F^{\pm}) \subset H$.  On $H$ the operator $\F$ coincides with $F$, so this completes the proof.
\end{proof}

\begin{corollary}
\label{cor:AssemblyAsIndex}
Let $\ell>0$ and let $\A$ be a Dixmier-Douady bundle on $G$ with $\tn{DD}(\A)=\ell \in \bZ \simeq H^3_G(G,\bZ)$.  Let $x=(\Phi,\S)_\ast[\scr{D}^E] \in \K_0^G(G,\A)$ be the class represented by a D-cycle $(M,E,\Phi,\S)$.  The formal character $\scr{I}(x) \in R^{-\infty}(T)^{W_\aff-\anti, \, \ell}$ is given by the $T$-index of a $1^{st}$ order elliptic operator $\st{D}^E_\beta$ acting on sections of a vector bundle $\wedge \n_- \wh{\otimes} S \otimes E$ over the space $\Y=\t \times_T \Phi^{-1}(U)$, where $U \supset T$ is a tubular neighborhood of the maximal torus.
\end{corollary}

\subsection{Application to Hamiltonian loop group spaces.}\label{sec:HamLGSpace}
A proper Hamiltonian $LG$-space $(\M,\omega_\M,\Phi_{\M})$ is a Banach manifold $\M$ with a smooth action of $LG$, equipped with a weakly non-degenerate $LG$-invariant closed 2-form $\omega_{\M}$, and a proper $LG$-equivariant map
\[ \Phi_\M \colon \M \rightarrow L\g^\ast \]
satisfying the moment map condition
\[ \iota(\xi_\M)\omega_\M=-d\pair{\Phi_\M}{\xi}, \qquad \xi \in L\g.\]
A \emph{level $k$ prequantization} of $\M$ is a $LG^{\bas}$-equivariant prequantum line bundle $L \rightarrow \M$, such that the central circle in $LG^{\bas}$ acts with weight $k$.  See for example \cite{MWVerlindeFactorization,AlekseevMalkinMeinrenken} for further background on Hamiltonian loop group spaces.

The subgroup $\Omega G \subset LG$ acts freely on $\M$, hence the quotient $M=\M/\Omega G$ is a smooth finite-dimensional $G$-manifold fitting into a pullback diagram
\begin{equation}
\begin{CD}
\M @>\Phi_{\M}>> L\g^\ast\\
@VVV      @VVV\\
M@>\Phi >>G   
\end{CD}
\end{equation}
where the vertical maps are the quotient maps by $\Omega G$.  The quotient $M$ is a \emph{quasi-Hamiltonian} (or \emph{q-Hamiltonian}) $G$-\emph{space}, and the pullback diagram above gives a 1-1 correspondence between proper Hamiltonian $LG$-spaces and compact q-Hamiltonian $G$-spaces \cite{AlekseevMalkinMeinrenken}.  

Let $G$ be compact and connected.  It was shown in \cite{DDDFunctor} (see \cite{LMSspinor} for a simpler construction) that every q-Hamiltonian $G$-space gives rise, in a canonical way, to a D-cycle $(M,\bC,\Phi,\S_{\tn{spin}})$ for $\K_0^G(G,\A)$ for a suitable Dixmier-Douady bundle $\A$ over $G$; the Morita morphism $\S_{\tn{spin}}$ is referred to as a \emph{twisted spin-c structure} in \cite{DDDFunctor,MeinrenkenKHomology,LMSspinor}.  For $G$ simple and simply connected, the Dixmier-Douady class of $\A$ is $\hvee \in \bZ \simeq H^3_G(G,\bZ)$, and we denote it by $\A^{(\hvee)}$.  We will assume $G$ is simple and simply connected below.

A \emph{level $k$ prequantization} \cite{MeinrenkenKHomology} of a q-Hamiltonian space is a Morita morphism
\[ \E \colon \underline{\bC} \dashrightarrow \Phi^\ast \A^{(k)} \]
where $\tn{DD}(\A^{(k)})=k \in \bZ \simeq H^3_G(G,\bZ)$.  Isomorphism classes of level $k$ prequantizations $\E$ of $M$ are in 1-1 correspondence with isomorphism classes of level $k$ prequantum line bundles $L$ over $\M$, see \cite{MeinrenkenKHomology,ZohrehPrequant} and references therein. 

Let $\S=\S_{\tn{spin}}\otimes \E$, then $(M,\bC,\Phi,\S)$ is a D-cycle for $\K_0^G(G,\A^{(k+\hvee)})$.  The \emph{level $k$ quantization} of $(M,\E)$ was defined by Meinrenken in \cite{MeinrenkenKHomology} as the image of the D-cycle $(M,\bC,\Phi,\S)$ in the analytic twisted K-homology group:
\begin{equation} 
\label{eqn:MeinrenkenDefinition}
(\Phi,\S)_\ast[\scr{D}] \in \K_0^G(G,\A^{(k+\hvee)}). 
\end{equation}
In light of the Freed-Hopkins-Teleman theorem, as well as the 1-1 correspondence between q-Hamiltonian $G$-spaces and Hamiltonian $LG$-spaces, it would seem reasonable to \emph{define} the level $k$ `quantization' of the prequantized loop group space $(\M,\omega_\M,\Phi_\M,L)$ as the element of $R_k(G)$ corresponding to $(\Phi,\S)_\ast [\scr{D}]$ under the Freed-Hopkins-Teleman isomorphism.  This definition satisfies many desirable properties.  For example, the quantization of a prequantized integral coadjoint orbit is the corresponding irreducible positive energy representation.  Also, the definition satisfies a `quantization commutes with reduction' principle, see \cite{MeinrenkenKHomology}.

In \cite{LSQuantLG}, building on constructions in \cite{LMSspinor}, we suggested an alternative definition of the quantization of a Hamiltonian loop group space in terms of the $T$-equivariant $L^2$-index of a Dirac-type operator on a non-compact spin-c submanifold of $\M$.  The latter submanifold and operator can be identified, respectively, with the manifold $\Y$ and the operator $\st{D}_\beta$ that we discussed in Section \ref{sec:IndexMap}; see \cite{LSQuantLG} for details.  As mentioned earlier, we proved in \cite[Section 4.7]{LSQuantLG} that $\st{D}_\beta$ has a well-defined $T$-equivariant $L^2$-index, with formal character lying in $R^{-\infty}(T)^{W_\aff-\anti,\, (k+\hvee)}$, and proposed that the quantization of $\M$ be \emph{defined} as the corresponding element the Verlinde ring $R_k(G)$.  The following is now an immediate consequence of Corollary \ref{cor:AssemblyAsIndex} and Proposition \ref{prop:InverseFHT}.
\begin{corollary}
\label{cor:DefsAgree}
The two definitions of the quantization of $\M$ agree, that is, under the identification $R^{-\infty}(T)^{W_\aff-\anti,\,(k+\hvee)}\simeq R_k(G)$, the $T$-equivariant $L^2-\index(\st{D}_\beta)$ coincides with the image of $(\Phi,\S)_\ast[\scr{D}] \in \K^G_0(G,\A^{(k+\hvee)})$ under the Freed-Hopkins-Teleman isomorphism.
\end{corollary}

Our principal motivation in \cite{LSQuantLG} was to give a definition amenable to study with the Witten deformation/non-abelian localization, and using this to obtain a new proof of the quantization-commutes-with-reduction theorem for Hamiltonian loop group spaces.  This was mostly carried out in \cite{LSWittenDef} (combined with certain results of \cite{YiannisThesis} or \cite{LMVerlindeQR}).  Thus a consequence of Corollary \ref{cor:DefsAgree} is that this new proof applies also to Meinrenken's \cite{MeinrenkenKHomology} definition \eqref{eqn:MeinrenkenDefinition}.

\appendix

\section{The $\KK$-product}
In this appendix we use the same notation as Section \ref{sec:IndexMap}: $(M,E,\Phi,\S)$ is a D-cycle representing a class $x=[\S]\otimes [\scr{D}^E] \in \K^G_0(X,\A)$.  We provide proofs of two results that were omitted.  These results are well-known at least in the case of a finite-rank Clifford module, as in the standard example \ref{ex:DeRhamDirac}, and the proofs are essentially the same as that case.  The first is Proposition \ref{prop:HilbertSpaceIso1}, which we restate here for the reader's convenience.

\begin{proposition}
\label{prop:HilbertSpaceIso2}
There is an isomorphism
\[ C_0(\S)\wh{\otimes}_{\Cl(M)} L^2(M,\Cliff(TM)\otimes E) \simeq L^2(M,\S\otimes E)\]
of $\bZ_2$-graded representations of $C_0(\A)$.
\end{proposition}
\begin{proof}
Let $C_0(\S)\odot_{\Cl(M)} L^2(M,\Cliff(TM)\otimes E)$ denote the algebraic graded tensor product of $\Cl(M)$-modules.  Define a pre-inner product on $C_0(\S)\odot_{\Cl(M)} L^2(M,\Cliff(TM)\otimes E)$ by the formula:
\begin{equation} 
\label{eqn:HilbModProd}
\pair{s_1\odot \varphi_1}{s_2\odot \varphi_2}=\big(\varphi_1,(s_1,s_2)_{\Cl(M)}\cdot \varphi_2\big)_{L^2},
\end{equation}
where $(s_1, s_2)_{\Cl(M)}$ denotes the $\Cl(M)$-valued inner product of the right Hilbert $\Cl(M)$-module $C_0(\S)$, and $(-,-)_{L^2}$ denotes the ordinary Hilbert space inner product on $L^2(M,\Cliff(TM)\otimes E)$.  Dividing by elements of length $0$ for the corresponding norm and then completing, we obtain a Hilbert space, usually denoted $C_0(\S)\otimes_{\Cl(M)}L^2(M,\Cliff(TM)\otimes E)$.  Using the action of $\Cliff(TM)$ on $\S$, there is a map
\[ C_0(\S)\odot_{\Cl(M)} L^2(M,\Cliff(TM)\otimes E) \rightarrow L^2(M,\S\otimes E) \]
with dense range.  The map intertwines \eqref{eqn:HilbModProd} with the inner product on $L^2(M,\S\otimes E)$, hence extends to an isomorphism from the completion to $L^2(M,\S\otimes E)$.
\end{proof}

Kasparov's fundamental class $[\scr{D}] \in \KK_G(\Cl(M),\bC)$ is the class defined by the operator $\scr{D}=d+d^\ast$ in the Hilbert space $L^2(M,\wedge T^\ast M)$ (cf. \cite[Definition 4.2]{KasparovNovikov}).  Identify $TM \simeq T^\ast M$ using the Riemannian metric.  In terms of a local orthonormal frame $e_i$, $i=1,...,\dim(M)$, the operator $d+d^\ast$ is given by
\[ \sum_i (\epsilon(e_i)-\iota(e_i))\nabla_{e_i} \]
where $\nabla$ is the Levi-Civita connection, and $\epsilon(v)$ denotes exterior multiplication by $v$ \cite[Lemma 5.13]{LawsonMichelsohn}.
There is a (unique) isomorphism of left $\Cliff(TM)$-modules
\[ \Cliff(TM) \rightarrow \wedge T^\ast M \]
which sends $1$ to $1$ and intertwines left multiplication on $\Cliff(TM)$ by $v \in TM$ with $\epsilon(v)+\iota(v) \in \End(\wedge T^\ast M)$.  Under this isomorphism $\epsilon(v)-\iota(v) \in \End(\wedge T^\ast M)$ corresponds to the endomorphism $\wh{v}$ of $\Cliff(TM)$ given by (cf. \cite[Section 1.11,1.12]{GuentnerHigson})
\[ \wh{v}\varphi = (-1)^{\deg(\varphi)}\varphi v. \]
Note that $\wh{v}^2=-\|v\|^2$.  Hence $\scr{D}=d+d^\ast$ corresponds to the operator in $L^2(M,\Cliff(TM))$ given in terms of a local orthonormal frame by the expression
\begin{equation} 
\label{eqn:LocalDeRhamDirac}
\sum_i \hc(e_i)\nabla_{e_i}. 
\end{equation}
More invariantly, the operator $\scr{D}$ (viewed as an operator in $L^2(M,\Cliff(TM))$) is given by the composition
\[ \Gamma^\infty(\Cliff(TM)) \xrightarrow{\nabla} \Gamma^\infty(T^\ast M \otimes \Cliff(TM)) \xrightarrow{\hc} \Gamma^\infty(\Cliff(TM)).\]

Recall from Section \ref{sec:IndexMap} that the candidate Dirac operator $\st{D}^E$ acting on smooth sections of $\S \otimes E$ is the composition
\begin{equation} 
\label{eqn:DefDE}
\Gamma^\infty(\S\otimes E) \xrightarrow{\nabla^{\S\otimes E}} \Gamma^\infty(T^\ast M \otimes \S \otimes E) \xrightarrow{g^\sharp} \Gamma^\infty(TM\otimes \S \otimes E) \xrightarrow{\hc} \Gamma^\infty(\S\otimes E). 
\end{equation}

\begin{theorem}
The cycle $(L^2(M,\S\otimes E),\rho,\st{D}^E)$ represents the class $[\S]\otimes [\scr{D}^E] \in \K^G_0(X,\A)$.
\end{theorem}
\begin{proof}
The presence of a vector bundle $E$ does not alter the proof, so we set $E=\bC$ to simplify notation.  We have shown that the triple $(L^2(M,\S),\rho,\st{D})$ represents a class in $\K^G_0(X,\A)$ (see Proposition \ref{prop:TwistedKHomCyc}) with the correct Hilbert space and representation.  Thus it suffices to check the product criterion in unbounded KK-theory \cite{KucerovskyUnbounded}, which involves checking a `connection condition' and a `semi-boundedness condition'.  The semi-boundedness condition is automatically satisfied, because the operator in the triple representing $[\S]$ is $0$.

For $s \in C_0(\S)$, let $T_s$ denote the map
\[ \varphi \in L^2(M,\Cliff(TM)) \mapsto s\otimes \varphi \in C_0(\S)\otimes_{\Cl(M)}L^2(M,\Cliff(TM)).\]
The `connection condition' says that for a dense set of $s \in C_0(\S)$ the operators
\begin{equation} 
\label{eqn:ConnectionRelations}
\st{D}\circ T_s - (-1)^{\deg(s)}T_s\circ \scr{D}, \qquad T_s^\ast \circ \st{D}-(-1)^{\deg(s)} \scr{D} \circ T_s^\ast 
\end{equation}
extend to bounded operators from $L^2(M,\Cliff(TM))$ to $L^2(M,\S)$.  Let $\varphi \in \Gamma^\infty(\Cliff(TM))$.  

From Proposition \ref{prop:HilbertSpaceIso2},
\begin{equation} 
\label{eqn:Pf1}
C_0(\S)\otimes_{\Cl(M)}L^2(M,\Cliff(TM)) \simeq L^2(M,\S),
\end{equation}
and
\[ T_s(\varphi)=\c(\varphi)s.\]
Calculating in terms of a local orthonormal frame and using \eqref{eqn:CliffConnection} we have  
\begin{align*} 
\st{D}\circ T_s(\varphi)&=\sum_i \hc(e_i)\nabla^{\S}_{e_i}(\c(\varphi)s)\\
&=(-1)^{\deg(s)+\deg(\varphi)}\sum_i \c(e_i)\nabla^{\S}_{e_i}(\c(\varphi)s)\\
&=(-1)^{\deg(s)+\deg(\varphi)}\sum_i \c(e_i)\c(\nabla_{e_i}\varphi)s+\c(e_i)\c(\varphi)\nabla^{\S}_{e_i}s.
\end{align*}
The second term is bounded (in $\varphi$).  For the first term recall that $\c$ is a right action, hence
\[ (-1)^{\deg(\varphi)}\c(e_i)\c(\nabla_{e_i}\varphi)=(-1)^{\deg(\varphi)}\c\big((\nabla_{e_i}\varphi) e_i\big)=\c(\wh{e}_i\nabla_{e_i}\varphi). \]
Thus, using \eqref{eqn:LocalDeRhamDirac}, the first term is
\[ (-1)^{\deg(s)}\c(\scr{D} \varphi)s=(-1)^{\deg(s)}T_s \circ \scr{D}(\varphi). \]
The argument for $T_s^\ast$ is similar.
\end{proof}

\bibliographystyle{amsplain}
\bibliography{../Biblio}
\end{document}